\documentclass[a4paper,11pt]{article}
\usepackage[utf8]{inputenc}
\usepackage{graphicx}
\usepackage[margin=1in]{geometry}
\usepackage{bbm}
\usepackage{amssymb,amsfonts,amsmath,amsfonts,amsthm}
\newtheorem{theorem}{Theorem}[section]
\newtheorem{corollary}{Corollary}[theorem]
\newtheorem{lemma}{Lemma}[section]
\newtheorem{proposition}{Proposition}[section]
\theoremstyle{remark}
\newtheorem{remark}{Remark}[section]
\newcommand{\R}{\mathbb R}

\newcommand{\E}{\mathbf {E}}

\newcommand{\N}{\mathbb {N}}

\title{Fractional integrals, derivatives and integral equations with weighted Takagi–Landsberg functions}
\author{Vitalii Makogin, Yuliya Mishura}
\date{\today}

\begin{document}

\maketitle

\abstract{In this paper  we find fractional Riemann–Liouville derivatives for the Takagi-Landsberg functions. Moreover, we introduce their generalizations called weighted Takagi-Landsberg functions which have arbitrary bounded coefficients in the expansion under Schauder basis. The class of the weighted Takagi-Landsberg functions of order $H>0$ on $[0,1]$ coincides with the $H$-H\"{o}lder continuous functions on $[0,1]$. Based on computed fractional integrals and derivatives of the Haar and Schauder functions, we get a new series representation of the fractional derivatives of a H\"{o}lder continuous function. This result allows to get the new formula of a Riemann-Stieltjes integral. The application of such series representation is the new method of numerical solution of the Volterra and linear integral equations driven by a H\"{o}lder continuous function.}
\section{Introduction}

The aim of this paper is to get a broad class of continuous functions on [0,1], which are nowhere differentiable but have fractional derivatives.
The prominent example is the Takagi–Landsberg function with Hurst parameter $H>0$ introduced in \cite{landsberg}, given by
\begin{equation}
\label{defKL}x^H(t)=\sum_{m=0}^\infty 2^{m\left(\frac{1}{2}-H\right)}\sum_{k=0}^{2^m-1}e_{m,k}(t),t\in [0,1],
\end{equation}
where $\{e_{m,k},m\in \N_0,k=0,\ldots,2^m-1\}$ are the Faber-Schauder functions on [0,1].
In the present paper, we find the fractional derivatives of the Takagi-Landsberg functions, and for other properties we refer to the surveys \cite{allaart} and \cite{Lagarias}.
In the case $H=1/2,$ the function $x^H$ is  known as the Takagi function.

There are several generalizations of the  function $x^H.$ In the paper of Mishura and Schied \cite{ms}, the signed Takagi-Landsberg functions of the form $$\sum_{m=0}^\infty 2^{m\left(\frac{1}{2}-H\right)}\sum_{k=0}^{2^m-1}\theta_{m,k}e_{m,k}(t),t\in [0,1] \text{ with } \theta_{m,k}\in\{-1,+1\}$$ are considered. Their results concern the maximum, the maximizers, and the modulus of continuity. Particularly, it was shown that $\max_{t\in[0,1]}x^H(t)=\frac{1}{3(1-2^{-H})}.$ The case of $H=1/2$ is considered in \cite{Schied}, where the connections to the F\"{o}lmer's pathwise It\^{o} calculus (e.g. \cite{folmer}) is also described.

In the present paper we go further and introduce so-called weighted Takagi-Landsberg functions, for which we let $\theta_{m,k}$ be arbitrary bounded coefficients.  We show that such weighted Takagi-Landsberg functions coincides with the H\"{o}lder continuous functions which immediately gives the new series representation for them, which we call a Takagi-Landsberg representation. Then we compute the fractional Riemann–Liouville derivatives and integrals of the Faber-Schauder functions, and therefore we obtain the fractional derivatives of the (weighted) Takagi-Landsberg functions. Such a new series representation of the fractional derivative for H\"{o}lder continuous functions is very promising for further development of the continuous functions without derivatives. Particularly, the Takagi-Landsberg representation gives the new method for numerical solution of the integral equations involving H\"{o}lder continuous functions.

As an example, we consider the Volterra integral equation with fractional noise, called also fractional Langevin equation, e.g. \cite{pironi,Fa}. This equation is of interest for modeling of anomalous diffusion in physics (e.g. \cite{Lutz}, \cite{kobelev}) and financial markets (e.g. \cite{west}).
Our method of its numerical solution allows to reduce it to  the system of linear algebraic equations, which is computationally effective. We prove that the numerical solution of the fractional Langevin equation, due to our method, approaches the theoretical solution, which is illustrated by numerical examples.

We obtain also the series expansion of the Riemann-Stieltjes integral applying methodology based on fractional Rieman-Liuville integrals introduced in \cite{Zahle} and developed in \cite{nualart}. As an illustration we consider the linear differential equation driven by H\"{o}lder continuous function and prove that its numerical solution due to our method tends to the exact solution in the specific norm. This result are supported also by numerical examples.

The paper is organized as follows. In Section 2, we recal some basic definitions from fractional calculus and Schauder basis. In Section 3, we compute fractional Riemann–Liouville integrals and derivatives of the Haar (Section 3.1) and the Faber-Schauder (Section 3.2) functions. In Section 4, we introduce the weighted Takagi-Landsberg functions and obtain the series representations of their Riemann–Liouville  derivatives. The series expansion of the Riemann-Stieltjes integral is given in Section 5. In Section 6, we consider the application of the Takagi-Landsberg representation for the solution of the Volterra integral (Section 6.1) and linear differential (Section 6.2) equations. The numerical results are presented in Sections 6.3 and 6.4.

\section{Preliminaries}
First we recall the definitions of fractional Riemann-Liouville   integrals and derivatives and their basic properties.
Let $f\in L_1([0,T]).$ We define left- and right-sided fractional integrals of order $\alpha>0$ on $(0,T)$  by
\begin{align}
  \label{frint1}  [I_{0+}^{\alpha}f](t)&:=\frac{1}{\Gamma(\alpha)}\int_0^{t}(t-u)^{\alpha-1}f(u)du,\\
 \label{frint2}     [I_{T-}^{\alpha}f](t)&:=\frac{1}{\Gamma(\alpha)}\int_t^{T}(u-t)^{\alpha-1}f(u)du,
\end{align}
respectively (cf. \cite[Definition 2.1]{skm}).

Define the spaces of functions that can be represented as fractional integrals:
$$I_+^{\alpha}(L_p([0,T])):=\{f\in L_1([0,T]):\exists \varphi\in L_p([0,T]) \text{ such that } f=I_{0+}^{\alpha}\varphi\},$$
$$I_-^{\alpha}(L_p([0,T])):=\{f\in L_1([0,T]):\exists \varphi\in L_p([0,T]) \text{ such that } f=I_{T-}^{\alpha}\varphi\}.$$
From \cite[formula (2.19)]{skm}  it follows that  $I_+^{\alpha}(L_p([0,T]))=I_-^{\alpha}(L_p([0,T]))$ for $1<p<\frac{1}{\alpha}.$

For the functions from $I_+^{\alpha}(L_1([0,T]))=I_-^{\alpha}(L_1([0,T]))$ we define
the left- and right-sided fractional Riemann-Liouville  derivatives on $(0,T)$ of order $\alpha$ by
\begin{align}
 \label{frder2}   [D_{0+}^{\alpha}f](t)&=\frac{1}{\Gamma(1-\alpha)}\frac{d}{dt}\int_0^{t}(t-u)^{-\alpha}f(u)du\\
\label{frder3}    [D_{T-}^{\alpha}f](t)&=-\frac{1}{\Gamma(1-\alpha)}\frac{d}{dt}\int_t^{T}(u-t)^{-\alpha}f(u)du.
\end{align}

Recall that the Faber–Schauder functions are defined as
$$e_{\emptyset}(t) := t,\quad e_{0,0}(t) :=(\min\{t, 1 -t\})^+, \quad e_{m,k}(t) :=2^{-m/2}e_{0,0}(2^m t - k)$$
for $t\in \R, m\in \N, k \in \N_0.$
They can be expressed in terms of Haar functions $H_{m,k}$ as
\begin{equation}
    e_{m,k}(t)=\int_0^tH_{m,k}(s)ds= [I_{0+}^1H_{m,k}](t), \text{ where }
    H_{m,k}(s)=2^{\frac{m}{2}}\mathbbm{1}_{J_{m,k}}(s)-2^{\frac{m}{2}}\mathbbm{1}_{J_{m,k+0.5}}(s),
\end{equation}
 and $J_{m,t}:=\left(\frac{t}{2^{m}},\frac{t+0.5}{2^{m}}\right],$ $m\in \N_0,k=0,\ldots,2^m-1,t\in [0,1].$

The Faber-Schauder functions form a Schauder basis in $C([0,1])$ and produce the following expansion of a function $f\in C([0,1])$ (e.g. \cite{kashin})
\begin{align}
\label{eq:expan}f(t)&=f(0)+(f(1)-f(0))t+\sum_{m=0}^\infty \sum_{k=0}^{2^m-1} 2^{\frac{m}{2}} a_{m,k}e_{m,k}(t), t\in [0,1],
\end{align}
with coefficients
\begin{equation*}
    \label{coef1} a_{m,k}=2f\left(\frac{k+0.5}{2^m}\right)-f\left(\frac{k+1}{2^m}\right)-f\left(\frac{k}{2^m}\right).
\end{equation*}

\section{Fractional derivatives of the Takagi–Landsberg function}

\subsection{Haar functions}
In this section, we calculate the fractional integrals and derivatives of the Haar functions.

\begin{lemma}
\label{haarintegr}Let $\alpha>0,$ $T> 0,$ $k,m\in \N_0$ and $0\leq k<2^m.$ Then for $t\in (0,1)$ we have
\begin{equation}
  \label{Ih0mk} I_{0+}^{\alpha}H_{m,k}(t)=\frac{2^{\frac{m}{2}}}{\Gamma(1+\alpha)}\left(\left(t-\frac{k}{2^m}\right)^{\alpha}_+-2\left(t-\frac{k+0.5}{2^m}\right)^{\alpha}_+ +\left(t-\frac{k+1}{2^m}\right)^{\alpha}_+\right),
\end{equation}
 and
\begin{equation}
  \label{IhTmk}
I_{T-}^{\alpha}H_{m,k}(t)=\frac{2^{\frac{m}{2}}}{\Gamma(1+\alpha)}\left(2\left(T\wedge \frac{k+0.5}{2^m}-t\right)_+^{\alpha}-\left(T\wedge \frac{k}{2^m}-t\right)_+^{\alpha}-\left(T\wedge \frac{k+1}{2^m}-t\right)_+^{\alpha}\right),
\end{equation}
for $t\in (0,T).$
\end{lemma}
\begin{proof}
If $t<\frac{k}{2^m},$ then $I_{0+}^{\alpha}H_{m,k}(t)=0.$ Let $t\in J_{m,k},$ then
\begin{align}
  \label{int:h1:1}  I_{0+}^{\alpha}H_{m,k}(t)=\frac{2^{\frac{m}{2}}}{\Gamma(\alpha)}\int_{k/2^m}^{t}(t-u)^{\alpha-1}du=\frac{2^{\frac{m}{2}}}{\Gamma(1+\alpha)}\left(t-\frac{k}{2^{m}}\right)^{\alpha}.
\end{align}
Let $t\in J_{m,k+0.5},$ then
\begin{align}
\nonumber   I_{0+}^{\alpha}H_{m,k}(t)&=\frac{2^{\frac{m}{2}}}{\Gamma(\alpha)}\left(\int_{k/2^m}^{(k+0.5)/2^m}(t-u)^{\alpha-1}du-\int_{(k+0.5)/2^m}^{t}(t-u)^{\alpha-1}du\right)\\
  \label{int:h1:2} &=\frac{2^{\frac{m}{2}}}{\Gamma(1+\alpha)}\left(\left(t-\frac{k}{2^{m}}\right)^{\alpha}-2\left(t-\frac{k+0.5}{2^{m}}\right)^{\alpha}\right).
\end{align}
If $t> \frac{k+1}{2^{m}},$ then
\begin{align}
\nonumber I_{0+}^{\alpha}H_{m,k}(t)&=\frac{2^{\frac{m}{2}}}{\Gamma(\alpha)}\left(\int_{k/2^m}^{(k+0.5)/2^m}(t-u)^{\alpha-1}du-\int_{(k+0.5)/2^m}^{(k+1)/2^m}(t-u)^{\alpha-1}du\right)\\
  \label{int:h1:3}   &=\frac{2^{\frac{m}{2}}}{\Gamma(1+\alpha)}\left(\left(t-\frac{k}{2^{m}}\right)^{\alpha}-2\left(t-\frac{k+0.5}{2^{m}}\right)^{\alpha}+\left(t-\frac{k+1}{2^{m}}\right)^{\alpha}\right).
\end{align}
Summarizing \eqref{int:h1:1}--\eqref{int:h1:3}, we get the statement \eqref{Ih0mk}.

Now prove relation \eqref{IhTmk}.
Obviously, if $T<\frac{k}{2^m}$ or $t>\frac{k+1}{2^m},$ then $I_{T-}^{\alpha}H_{m,k}(t)=0.$
 Let $t\in J_{m,k+0.5},$ then
\begin{align}
  \label{int:h2:1}  I_{T-}^{\alpha}H_{m,k}(t)=-\frac{2^{\frac{m}{2}}}{\Gamma(\alpha)}\int_t^{T\wedge \frac{k+1}{2^m} }(u-t)^{\alpha-1}du=-\frac{2^{\frac{m}{2}}}{\Gamma(1+\alpha)}\left(T\wedge \frac{k+1}{2^m}-t\right)^{\alpha}.
\end{align}
Let $t\in J_{m,k},$ then
\begin{align}
\nonumber  I_{T-}^{\alpha}H_{m,k}(t)&=\frac{2^{\frac{m}{2}}}{\Gamma(\alpha)}\left(\int_t^{T\wedge \frac{k+0.5}{2^m} }(u-t)^{\alpha-1}du-\int_{T\wedge \frac{k+0.5}{2^m}}^{T\wedge \frac{k+1}{2^m} }(u-t)^{\alpha-1}du\right)\\
  \label{int:h2:2}  &=\frac{2^{\frac{m}{2}}}{\Gamma(1+\alpha)}\left(2\left(T\wedge \frac{k+0.5}{2^m}-t\right)^{\alpha}-\left(T\wedge \frac{k+1}{2^m}-t\right)^{\alpha}\right).
\end{align}
If $t< \frac{k}{2^{m}},$ then
\begin{align}
\nonumber    &I_{T-}^{\alpha}H_{m,k}(t)=\frac{2^{\frac{m}{2}}}{\Gamma(\alpha)}\left(\int_{T\wedge \frac{k}{2^m}}^{T\wedge \frac{k+0.5}{2^m} }(u-t)^{\alpha-1}du-\int_{T\wedge \frac{k+0.5}{2^m}}^{T\wedge \frac{k+1}{2^m} }(u-t)^{\alpha-1}du\right)\\
\label{int:h2:3}  &=\frac{2^{\frac{m}{2}}}{\Gamma(1+\alpha)}\left(-\left(T\wedge \frac{k}{2^m}-t\right)^{\alpha}+2\left(T\wedge \frac{k+0.5}{2^m}-t\right)^{\alpha}-\left(T\wedge \frac{k+1}{2^m}-t\right)^{\alpha}\right).
\end{align}
Summarizing \eqref{int:h2:1}--\eqref{int:h2:3}, we get the statement \eqref{IhTmk}.
\end{proof}

For $m\in \N_0,k=0,\ldots,2^m-1,H>0$ denote by
\begin{equation}
\label{taudef}\tau^{\alpha}_{1,2^m+k}(t)=\frac{\left( t-\frac{k}{2^m}\right)_+^{\alpha} - 2\left( t-\frac{k+0.5}{2^m}\right)_+^{\alpha}+\left( t-\frac{k+1}{2^m}\right)_+^{\alpha}}{\Gamma(1+\alpha)}, t\in [0,1],\alpha\geq 0,
\end{equation}
and
\begin{equation*}
\label{taudef2}\tau^{\alpha}_{2,2^m+k}(t,T)=\frac{2\left(T\wedge \frac{k+0.5}{2^m}-t\right)_+^{\alpha}-\left(T\wedge \frac{k}{2^m}-t\right)_+^{\alpha}-\left(T\wedge \frac{k+1}{2^m}-t\right)_+^{\alpha}}{\Gamma(1+\alpha)}, [0,T],\alpha\geq 0.
\end{equation*}

Then $I_{0+}^{\alpha}H_{m,k}(t)=2^{\frac{m}{2}}\tau^{\alpha}_{1,2^m+k}(t)$ and $I_{T-}^{\alpha}H_{m,k}(t)=2^{\frac{m}{2}}\tau^{\alpha}_{2,2^m+k}(t,T).$
\begin{remark}
\label{taub1}
We give immediate bounds for $\tau^{\alpha}_{1,2^m+k}$ and $\tau^{\alpha}_{2,2^m+k}.$ For instance, for any $m\in \N_0$ and $k=0,\ldots,2^{m-1}$ we have
\begin{align*}&|\tau^{\alpha}_{1,2^m+k}(t)|= |I_{0+}^{\alpha}2^{-\frac{m}{2}}H_{m,k}(t)|\leq I_{0+}^{\alpha}[\mathbbm{1}_{J_{m,k}}+\mathbbm{1}_{J_{m,k+0.5}}](t)\leq I_{0+}^{\alpha}[\mathbbm{1}_{J_{m,k}}+\mathbbm{1}_{J_{m,k+0.5}}](1)\\
&=\frac{1}{\Gamma(\alpha)}\int_{\frac{k}{2^m}}^{\frac{k+1}{2^m}}(1-u)^{\alpha-1}du=\frac{1}{\Gamma(1+\alpha)}\left(\left(1-\frac{k+1}{2^m}\right)^\alpha-\left(1-\frac{k}{2^m}\right)^\alpha\right)\leq \frac{2^{-m\alpha}}{\Gamma(1+\alpha)}.
\end{align*}
Similarly, we get that $|\tau^{\alpha}_{1,2^m+k}(t)|\leq I_{T-}^{\alpha}\mathbbm{1}_{J_{m,k}\cup J_{m,k+0.5}}(t)\leq  I_{1-}^{\alpha}\mathbbm{1}_{J_{m,k}\cup J_{m,k+0.5}}(0)\leq \frac{2^{-m\alpha}}{\Gamma(1+\alpha)}.$
\end{remark}
\begin{remark}
\label{fgn}
One can observe that functions $\tau^{\alpha}_{1,2^m+k}$ and $\tau^{\alpha}_{2,2^m+k}$ can be written in terms of a fractional Gaussian noise with Hurst index $H\in (0,1),$ that is a centered Gaussian process with  the covariance function
\begin{equation*}
    \label{Cdef}\E [Y^H(t)Y^H(0)]=C_H(t)=\frac{1}{2}\left(|t+1|^{2H}-2|t|^{2H}+|t-1|^{2H}\right),t\in \R.
\end{equation*}
Indeed,  $$\tau^{\alpha}_{1,2^m+k}(t)=\frac{2^{1-\alpha}}{2^{m \alpha}\Gamma(1+\alpha)} C_{\alpha/2}\left(2^{m+1}t-2k-1 \right), \text{ if }  t\geq \frac{k+1}{2^m}, $$
and
$$\tau^{\alpha}_{2,2^m+k}(t,T)=-\frac{2^{1-\alpha}}{2^{m \alpha}\Gamma(1+\alpha)} C_{\alpha/2}\left(2k+1-2^{m+1}t \right), \text{ if }  t\leq \frac{k}{2^m} \text{ and }  T\geq \frac{ k+1}{2^m}.$$

Since $\frac{\alpha}{2}\in \left(0,\frac{1}{2}\right)$ if $\alpha\in (0,1),$ we can study properties of the integrals $I_{0+}^{\alpha}H_{m,k}$ and $I_{T-}^{\alpha}H_{m,k}$ using the known results about $C_H$ with $H<1/2.$ For instance, it is known that $C_H(t)< 0$ if $t\geq 1$ and $H\in (0,1/2).$

Further, we use the fact that function $C_H$ in the case $H<1/2$ is absolutely integrable and monotonically increasing on $[1,+\infty),$ e.g. \cite[Section 3.2]{embr}
\end{remark}

\begin{remark}
\label{taub2}
We provide some auxiliary bounds for functions  $\tau^{\alpha}_{1,2^m+k}$ and $\tau^{\alpha}_{2,2^m+k}.$ Let $\alpha\in (0,1),$ then $C_{\frac{\alpha}{2}}(x)$ is  negative and monotonically increasing for $x\geq 1,$ which gives that $|C_{\frac{\alpha}{2}}(x)|\leq |C_{\frac{\alpha}{2}}(\left\lfloor x\right\rfloor)|, x\geq 1.$
Therefore,  $\tau^{\alpha}_{1,2^m+k}(t)$ is negative for $k\leq \left\lfloor 2^m t\right\rfloor-1$  and
\begin{align*}
    \frac{2^{m\alpha}\Gamma(1+\alpha)}{2^{1-\alpha}}|\tau^{\alpha}_{1,2^m+k}(t)|= \left|C_{\alpha/2}\left(2^{m+1}t-2k-1 \right)\right|\leq \left|C_{\alpha/2}\left(2 \left\lfloor 2^m t\right\rfloor-2k-1 \right)\right|.
\end{align*}
Similarly, if  $\left\lfloor{2^m t}\right\rfloor+1\leq k\leq \left\lfloor{2^m T}\right\rfloor-1,$  then
\begin{align*}
    \frac{2^{m\alpha}\Gamma(1+\alpha)}{2^{1-\alpha}}|\tau^{\alpha}_{2,2^m+k}(t,T)|= \left|C_{\alpha/2}\left(1+2k-2^{m+1}t \right)\right|\leq \left|C_{\alpha/2}\left(2k- \left\lfloor 2^m t\right\rfloor-1 \right)\right|.
\end{align*}
\end{remark}
\subsection{The Faber-Schauder functions}
Here, we find the fractional integrals and derivatives of the Faber-Schauder functions.
\begin{lemma}
\label{lemmaIe}
Let $\alpha\in (0,1),$ $T>0,$ $k,m\in \N_0$ and $0\leq k<2^m.$ Then for $t\in (0,1)$ we have
$I_{0+}^{\alpha}e_{m,k}(t)=I_{0+}^{1+\alpha}H_{m,k}(t)$ and $I_{T-}^{\alpha}e_{m,k}(t)=e_{m,k}(T)I_{T-}^{\alpha}\mathbbm{1}_{[0,1]}(t)- I_{T-}^{1+\alpha}H_{m,k}(t)$, $t\in(0,T).$
\end{lemma}
\begin{proof}
It follows from \cite[formula (2.65)]{skm} that
$I_{0+}^{\alpha}e_{m,k}=I_{0+}^{\alpha}I_{0+}^{1}e_{m,k}=I_{0+}^{1+\alpha}H_{m,k}.$ Consider $I_{T-}^{\alpha}e_{m,k}=I_{T-}^{\alpha}I_{0+}^{1}H_{m,k}.$ It equals
\begin{align*}
    &\frac{1}{\Gamma(\alpha)}\int_{t}^T\left(\int_{0}^s H_{m,k}(z)dz\right)(s-t)^{\alpha-1}ds =
    \frac{1}{\Gamma(\alpha)}\int_{0}^T H_{m,k}(z)\left(\int_{z\vee t}^T (s-t)^{\alpha-1}ds\right)dz\\
    &=\frac{1}{\Gamma(1+\alpha)}\int_{0}^T H_{m,k}(z)\left((T-t)^\alpha- (z-t)_+^{\alpha}ds\right)dz=\frac{e_{m,k}(T)}{\Gamma(1+\alpha)}(T-t)^\alpha- I_{T-}^{1+\alpha}H_{m,k}(t).
\end{align*}
Finally, we note that $\frac{1}{\Gamma(1+\alpha)}(T-t)^\alpha=I_{T-}^{\alpha}\mathbbm{1}_{[0,1]}(t)$
\end{proof}
\begin{proposition}
\label{propemk}
Let $\alpha\in (0,1),$ $T>0,$ $k,m\in \N_0$ and $0\leq k<2^m.$ Then for $t\in (0,1)$ we have
\begin{equation}
  \label{D0emk} D_{0+}^{\alpha}e_{m,k}(t)=\frac{2^{m(\alpha-\frac12)}}{\Gamma(2-\alpha)}\left((2^mt-k)^{1-\alpha}_+-2(2^mt-k-0.5)^{1-\alpha}_+ +(2^mt-k-1)^{1-\alpha}_+\right),
\end{equation}
 and
\begin{align}
  \label{DTemk}
&D_{T-}^{\alpha}e_{m,k}(t)={e_{m,k}(T)}D_{T-}^{\alpha}\mathbbm{1}_{[0,1]}(t)-I_{T-}^{1-\alpha}H_{m,k}(t), t\in(0,T).
\end{align}
\end{proposition}
\begin{proof}
Formula \eqref{D0emk} follows directly from Lemma \ref{lemmaIe} and formula \eqref{Ih0mk}, since $D_{0+}^{\alpha}e_{m,k}=D_{0+}^{\alpha}I_{0+}^{1}H_{m,k}=I_{0+}^{1-\alpha}H_{m,k},$ e.g. \cite[formula 2.65]{skm}.

We obtain the derivative $
D_{T-}^{\alpha}$ from the relation $D_{T-}^{\alpha}e_{m,k}(t)=-\frac{d}{dt}[I_{T-}^{1-\alpha}e_{m,k}](t).$ Thus, we have from Lemma \ref{lemmaIe} and formula \eqref{IhTmk} that
\begin{align*}&D_{T-}^{\alpha}e_{m,k}(t)=-\frac{d}{dt}\left(e_{m,k}(T)I_{T-}^{1-\alpha}\mathbbm{1}_{[0,1]}(t)- I_{T-}^{2-\alpha}H_{m,k}(t)\right)=\frac{e_{m,k}(T)}{\Gamma(1-\alpha)}(T-t)^{-\alpha}\\
&+\frac{d}{dt}I_{T-}^{1-\alpha}I_{T-}^{1}H_{m,k}(t)=\frac{e_{m,k}(T)}{\Gamma(1-\alpha)}(T-t)^{-\alpha}-D_{T-}^{\alpha}I_{T-}^{1}H_{m,k}(t)\\
&=e_{m,k}(T)D_{T-}^{\alpha}\mathbbm{1}_{[0,1]}(t)-I_{T-}^{1-\alpha}H_{m,k}(t).
\end{align*}
\end{proof}
\begin{remark}
We can write the fractional derivatives $D_{0+}^{\alpha}e_{m,k}$ and $D_{T-}^{\alpha}e_{m,k}$ as
\begin{equation}
\label{der:repr}    D_{0+}^{\alpha}e_{m,k}(t)=2^{\frac{m}{2}}\tau^{1-\alpha}_{1,2^m+k}(t), \quad D_{T-}^{\alpha}e_{m,k}(t)=\frac{e_{m,k}(T)}{\Gamma(1-\alpha)}(T-t)^{-\alpha}-2^{\frac{m}{2}}\tau^{1-\alpha}_{2,2^m+k}(t,T).
\end{equation}
\end{remark}

\begin{lemma}
\label{diff:parts}
1) Let a series $\sum_{n=0}^\infty a_n(t),t\in [0,T]$ be uniformly bounded by a non-negative function $A\in L_1[0,T],$ then
\begin{equation}
    I^\alpha_{0+}\left(\sum_{n=0}^\infty a_n \right)(t) = \sum_{n=0}^\infty \left(I^\alpha_{0+}a_n\right)(t), t\in [0,T].
\end{equation}

2) Let $\sum_{n=0}^\infty a_n(t),t\in [0,T]$ be a convergent in $L_1[0,T]$ series with $a_n\in I_+^\alpha(L_1[0,T]),n \geq 0.$ If the exists a summable sequence $b_n\geq 0,n\geq 0$ such that $|\left(D^\alpha_{0+}a_n\right)(t)|\leq b_n$ for all $t\in[0,T]$, then
\begin{equation}
     D^\alpha_{0+}\left(\sum_{n=0}^\infty a_n \right)(t) = \sum_{n=0}^\infty \left(D^\alpha_{0+}a_n\right)(t), t\in [0,T].
\end{equation}
\end{lemma}
\begin{proof}
1) The first statement follows from the Lebesgue  dominated convergence theorem, that is
$$ \int_0^t(t-z)^{\alpha-1}\left(\sum_{n=0}^\infty |a_n(z)|\right)dz \leq \int_0^t(t-z)^{\alpha-1}A(z)dz=\Gamma(\alpha)\left(I_{0+}^\alpha A\right) (t),$$
where $I_{0+}^\alpha A(t)$ is finite for almost all $t\in (0,T)$ due to $\|I^\alpha_{0+} A\|_{L_1[0,T]}<\infty,$ e.g \cite[Theorem 2.6]{skm}.

2) Note that  $\left(D_{0+}^\alpha a_n\right)(t)=\frac{d}{dt}\left(I^{1-\alpha}_{0+} a_n\right)(t)$. Since $\sum_{n=0}^\infty a_n\in L_1[0,T],$ we have from the first part that $I^{1-\alpha}_{0+}\left(\sum_{n=0}^\infty a_n  \right)(t)=\sum_{n=0}^\infty \left(I^{1-\alpha}_{0+} a_n  \right)(t).$  Then
\begin{align*}
     &D^\alpha_{0+}\left(\sum_{n=0}^\infty a_n \right)(t)=
     \frac{d}{dt} \left(I^{1-\alpha}_{0+}\left(\sum_{n=0}^\infty a_n  \right)\right)(t)= \frac{d}{dt} \left(\sum_{n=0}^\infty \left(I^{1-\alpha}_{0+} a_n \right)\right)(t).
\end{align*}
Since $\sum_{n=0}^\infty\left|\frac{d}{dt}\left(I^{1-\alpha}_{0+} a_n\right)(t)\right|\leq \sum_{n=0}^\infty b_n<\infty,$  we have
\begin{align*}
\frac{d}{dt} \left(\sum_{n=0}^\infty \left(I^{1-\alpha}_{0+} a_n \right)\right)(t)=\sum_{n=0}^\infty \frac{d}{dt} \left(I^{1-\alpha}_{0+} a_n \right)(t) =\sum_{n=0}^\infty \left(D^\alpha_{0+}a_n\right)(t), t\in [0,T].
\end{align*}
\end{proof}

Consider the partial sums of the fractional derivatives of the Faber-Schauder functions $D_{0+}^{\alpha}\left[\sum_{k=0}^{2^m-1}e_{m,k}\right]\left(t\right)$ and $D_{T-}^{\alpha}\left[\sum_{k=0}^{2^m-1}e_{m,k}\right]\left(t\right).$
Due to \eqref{DTemk} and \eqref{der:repr},  we have $$D_{0+}^{\alpha}\left[\sum_{k=0}^{2^m-1}e_{m,k}\right]\left(t\right)= 2^{\frac{m}{2}}\sum_{k=0}^{\left\lfloor{2^m t}\right\rfloor-1} \tau^{1-\alpha}_{1,2^m+k}(t)+2^{\frac{m}{2}}\tau^{1-\alpha}_{1,2^m+\left\lfloor{2^m t}\right\rfloor}(t)$$
and
$$D_{T-}^{\alpha}\left[\sum_{k=0}^{2^m-1}e_{m,k}\right]\left(t\right)-\left(\sum_{k=0}^{2^m-1}e_{m,k}(T)\right)D_{T-}^{\alpha}\mathbbm{1}_{[0,1]}\left(t\right)=- 2^{\frac{m}{2}}\sum_{k=\left\lfloor{2^m t}\right\rfloor}^{\left\lfloor{2^m T}\right\rfloor}\tau^{1-\alpha}_{2,2^m+k}(t,T).$$
\begin{proposition}
\label{uniformb} If $m\geq 1,$ then
\begin{align}
    \label{ineqD0emk1} \left|\sum_{k=0}^{2^m-1}\tau^{1-\alpha}_{1,2^m+k}(t)\right|&\leq c_1(\alpha) 2^{m(\alpha-1)},   \text{ uniformly on }[0,1], \\
    \label{ineqDTemk1} \left|\sum_{k=0}^{2^m-1}\tau^{1-\alpha}_{2,2^m+k}(t,T)\right|&\leq   c_1(\alpha)  2^{m(\alpha-1)},  \text{ uniformly on }[0,T],
\end{align}
where $c_1(\alpha)=\frac{2^\alpha}{\Gamma(2-\alpha)}\left(\sum_{k\geq 1} \left|C_{\frac{1-\alpha}{2}}\left(k \right)\right|+2 \right).$
\end{proposition}
\begin{proof}
From Remarks \ref{taub1} and \ref{taub2}, it follows that
\begin{align*}
&\left|\sum_{k=0}^{2^m-1}\tau^{1-\alpha}_{1,2^m+k}(t)\right|\leq \frac{2^{m(\alpha-1)+\alpha}}{\Gamma(2-\alpha)}\sum_{k=0}^{\left\lfloor{2^m t}\right\rfloor-1} \left|C_{\frac{1-\alpha}{2}}\left(2 \left\lfloor 2^m t\right\rfloor-2k-1 \right)\right|+\frac{2^{m(\alpha-1)}}{\Gamma(2-\alpha)}\\
&\leq \frac{2^{m(\alpha-1)+\alpha}}{\Gamma(2-\alpha)}\left(\sum_{k\geq 1} \left|C_{\frac{1-\alpha}{2}}\left(k \right)\right|+2^{-\alpha} \right)\leq c_1(\alpha)  2^{m(\alpha-1)},
\end{align*}
where the series $\sum_{k\geq 1}\left|C_{\frac{1-\alpha}{2}}\left(k \right)\right|$ converges due to integrability of $C_{\frac{1-\alpha}{2}},$ e.g. \cite[Section 3.2]{embr}.

Consider the case of $D_{T-}^{\alpha}.$
Let $\left\lfloor{2^m t}\right\rfloor+1\leq \left\lfloor{2^m T}\right\rfloor-1,$ then
it follows from Remarks \ref{taub1} and \ref{taub2} that
\begin{align*}
&\left|\sum_{k=0}^{2^m-1}\tau^{1-\alpha}_{2,2^m+k}(t,T)\right|\leq
\sum_{k=\left\lfloor{2^m t}\right\rfloor+1}^{\left\lfloor{2^m T}\right\rfloor-1} |\tau^{1-\alpha}_{2,2^m+k}(t,T)|+2\frac{2^{m(\alpha-1)}}{\Gamma(2-\alpha)}\\
&\leq
\frac{2^{m(\alpha-1)+\alpha}}{\Gamma(2-\alpha)} \sum_{k=\left\lfloor{2^m t}\right\rfloor+1}^{\left\lfloor{2^m T}\right\rfloor-1} \left|C_{\frac{1-\alpha}{2}}\left(2k-2 \left\lfloor 2^m t\right\rfloor-1 \right)\right|+2\frac{2^{m(\alpha-1)}}{\Gamma(2-\alpha)}\\
&\leq \frac{2^{m(\alpha-1)+\alpha}}{\Gamma(2-\alpha)}\left(\sum_{k\geq 1} \left|C_{\frac{1-\alpha}{2}}\left(k \right)\right|+2^{1-\alpha} \right).
\end{align*}
Let $\left\lfloor{2^m t}\right\rfloor+1 \geq \left\lfloor{2^m T}\right\rfloor.$ Then there are at most two non-zero $\tau^{1-\alpha}_{2,2^m+k}(t,T)$. Thus, we get the upper bound
$\left|\sum_{k=0}^{2^m-1}\tau^{1-\alpha}_{2,2^m+k}(t,T)\right|\leq
2\frac{2^{m(\alpha-1)}}{\Gamma(2-\alpha)}.$
\end{proof}

It follows from Proposition \ref{uniformb} that the series $\sum_{m\geq 0}2^{m\left(\frac{1}{2}-H\right)}\left|D_{0+}^{\alpha}\left[\sum_{k=0}^{2^m-1}e_{m,k}\right]\left(t\right)\right|$ converges uniformly on $[0,1]$ for $\alpha<H.$ This ensures that Lemma \ref{diff:parts} holds for the Takagi-Landsberg function $x^H$ and yields $$D_{0+}^\alpha x^H(t)=\sum_{m= 0}^{\infty}2^{m\left(1-H\right)}\sum_{k=0}^{2^m-1} \tau^{1-\alpha}_{1,2^m+k}(t).$$

Take the expansion $x^H(t)-x^H(T)=\sum_{m\geq 0}2^{m\left(\frac{1}{2}-H\right)}\sum_{k=0}^{2^m-1}(e_{m,k}(t)-e_{m,k}(T)).$ Since
the series $\sum_{m\geq 0}2^{m\left(\frac{1}{2}-H\right)}\left|D_{T-}^{\alpha}\left[\sum_{k=0}^{2^m-1}(e_{m,k}-e_{m,k}(T)\right]\left(t\right)\right|$ converges uniformly on $[0,T]$ for $\alpha<H,$ then it holds
by  Lemma \ref{diff:parts} that
$$D_{T-}^\alpha x^H(t)= x^H(T) D_{T-}^\alpha\mathbbm{1}_{[0,1]}(t) -\sum_{m= 0}^{\infty}2^{m\left(1-H\right)}\sum_{k=0}^{2^m-1} \tau^{1-\alpha}_{2,2^m+k}(t,T).$$

Now consider the special case $\alpha=H$ and the values of $D_{0+}^\alpha x^H(t)$ at points of the $m_0$th dyadic partition of $[0,1],$ that is the set $\mathbb{T}_{m_0}:=\{k 2^{-m_0}|k=0,\ldots, 2^{m_0}\}.$
\begin{proposition}
Let $k_0,m_0 \in \N_0$ and $k_0\leq 2^{m_0}-1.$ Then
$$\sum_{m=0}^\infty 2^{m\left(\frac{1}{2}-H\right)}\sum_{k= 0}^{2^m-1}D_{0+}^{H}e_{m,k}\left(\frac{k_0}{2^{m_0}}\right)=-\infty.$$
\end{proposition}
\begin{proof}
In the case $\alpha=H,$ $m\geq m_0,$ it follows from Remark \ref{fgn} that
\begin{align}
    \label{prop2:eq1}&d_m:=2^{m\left(\frac{1}{2}-H\right)}\sum_{k= 0}^{2^m-1}D_{0+}^{H}e_{m,k}\left(\frac{k_0}{2^{m_0}}\right)=\frac{2^{H}}{\Gamma(2-H)}\sum_{k=0}^{2^{m-m_0}k_0-1} C_{\frac{1-H}{2}}\left(\frac{2^{m+1}k_0}{2^{m_0}}-2k-1 \right).
\end{align}
For all $k\leq 2^{m-m_0}k_0-1$ we have $2^{m-m_0+1}k_0-2k-1\geq 1$ and
\begin{equation}
\label{negativC}    C_{\frac{1-H}{2}}\left(2^{m-m_0+1}k_0-2k-1 \right)< 0,
\end{equation} which gives that the right hand side of \eqref{prop2:eq1} is negative.

Now we show that the sequence $d_m$ is monotonically decreasing if $m\geq m_0.$
Consider the difference $d_{m+1}-d_m$, which equals
\begin{align*}
    &\frac{2^{H}}{\Gamma(2-H)}\left[\sum_{k=0}^{\frac{2^{m+1}k_0}{2^{m_0}}-1} C_{\frac{1-H}{2}}\left(\frac{2^{m+1+1}k_0}{2^{m_0}}-2k-1 \right)  -\sum_{k=0}^{\frac{2^{m}k_0}{2^{m_0}}-1} C_{\frac{1-H}{2}}\left(\frac{2^{m+1}k_0}{2^{m_0}}-2k-1 \right)\right]\\
&=\frac{2^{H}}{\Gamma(2-H)}\sum_{k=0}^{2^{m-m_0}k_0-1} C_{\frac{1-H}{2}}\left(2^{m-m_0+2}k_0-2k-1 \right).
\end{align*}
We get from the last relation and \eqref{negativC} that $d_{m+1}-d_m<0,$
so $d_{m+1}<d_{m}<d_{m_0}<0$ for all $m>m_0.$ This means that $\sum_{m=m_0}^\infty d_m<\sum_{m=m_0}^\infty d_{m_0}=-\infty.$ 

\end{proof}

\section{A weighted Takagi–Landsberg function}
In this section we consider the    extension of the class of the Takagi–Landsberg functions. Namely, for  constants $c_{m,k}\in [-L,L], k,m\in \N_0,$ we  define a {\it weighted Takagi–Landsberg function} as
$y_{c,H}:[0,1]\to \R$ via
\begin{equation}
    \label{ydef}
    y_{c,H}(t)=\sum_{m=0}^\infty 2^{m\left(\frac{1}{2}-H\right)}\sum_{k=0}^{2^m-1}{c_{m,k}}e_{m,k}(t), t\in [0,1].
\end{equation}
Since $|y_{c,H}(t)|\leq L x^H(t),t\in [0,1],$ the series in \eqref{ydef} converges uniformly and $y_{c,H}\in L_1([0,1]).$



\begin{lemma}
Let $H>0.$ Any $H$-H\"{o}lder continuous function $f$ on [0,1] can be expanded as
\begin{equation}
    \label{TLr}
    f(t)=f(0)(1-t)+f(1)t+\sum_{m=0}^\infty 2^{m\left(\frac{1}{2}-H\right)}\sum_{k=0}^{2^m-1}{c_{m,k}}e_{m,k}(t), t\in [0,1].
\end{equation}
We call formula \eqref{TLr} the Takagi-Landsberg representation of function $f.$
\end{lemma}
\begin{proof}
To show this, we first provide the relation between coefficients $a_{m,k}$ in expansion \eqref{eq:expan} and $c_{m,k}$ in \eqref{ydef} that is
\begin{align}
\label{eq:coef}    c_{m,k}=a_{m,k}2^{mH}=2^{mH}\left[2f\left(\frac{k+0.5}{2^m}\right)-f\left(\frac{k+1}{2^m}\right)-f\left(\frac{k}{2^m}\right)\right].
\end{align}

Theorem 3 on p. 191 in \cite{kashin} states that $f$ is $H$-H\"{o}lder continuous if and only if coefficients $a_{m,k}$ in expansion \eqref{eq:expan} satisfy $|a_{m,k}|\leq C (2^m+k)^{-H},$ $m\geq 0$ for a constant $C>0.$
Thus, if $f$ is $H$-H\"{o}lder continuous, then $|c_{m,k}|=|a_{m,k}|2^{mH}\leq C:=L$ and $f$ is a weighed  Takagi–Landsberg function. If $y_{c,H}$ admits representation \eqref{ydef}, i.e. $c_{m,k}\in[-L,L],$ then
$a_{m,k}=c_{m,k}2^{-mH}$ from \eqref{eq:coef} satisfy
$|a_{m,k}|\leq L 2^{-mH} \leq 2 L(2^m+k)^{-H}, m\geq 0.$ Hence, $y_{c,H}$ is $H$-H\"{o}lder continuous.
\end{proof}

Now let us  establish that $y_{c,H}$ admit fractional derivatives of order $\alpha<H$.
\begin{theorem}
\label{thmWF}
Let $0<\alpha<H$ then
\begin{align}
   &\label{D0y} D_{0+}^{\alpha} y_{c,H}(t) =\sum_{m=0}^\infty 2^{m(1-H)}\sum_{k=0}^{2^m-1}c_{m,k} \tau_{1,2^m+k}^{1-\alpha}(t)\\
&\label{DTy} D_{T-}^{\alpha} [y_{c,H}-y_{c,H}(T)](t)  =-\sum_{m=0}^\infty 2^{m(1-H)}\sum_{k=0}^{2^m-1}c_{m,k} \tau_{2,2^m+k}^{1-\alpha}(t,T).
\end{align}
\end{theorem}
\begin{proof}
Due to \eqref{D0emk}, the fractional derivatives of summands in \eqref{ydef} equal
\begin{align*}
  D_{0+}^\alpha\left[2^{m\left(\frac{1}{2}-H\right)}\sum_{k=0}^{2^m-1}{c_{m,k}}e_{m,k}\right](t)=2^{m(1-H)}\sum_{k=0}^{2^m-1}c_{m,k} \tau_{1,2^m+k}^{1-\alpha}(t) .
\end{align*}
From \eqref{ineqD0emk1} we have the following uniform bound
$$\left|D_{0+}^\alpha\left[2^{m\left(\frac{1}{2}-H\right)}\sum_{k=0}^{2^m-1}{c_{m,k}}e_{m,k}\right](t)\right|\leq L 2^{m(1-H)}\sum_{k=0}^{2^m-1}\left|\tau_{1,2^m+k}^{1-\alpha}(t)\right|\leq L c_1(\alpha) 2^{m(\alpha-H)}.$$
Analogously,
$$\left|D_{T-}^\alpha\left[2^{m\left(\frac{1}{2}-H\right)}\sum_{k=0}^{2^m-1}c_{m,k}(e_{m,k}-e_{m,k}(T))\right](t)\right|\leq L 2^{m(1-H)}\sum_{k=0}^{2^m-1}\left|\tau_{2,2^m+k}^{1-\alpha}(t)\right|\leq L c_1(\alpha) 2^{m(\alpha-H)}.$$
Thus, from Lemma \ref{diff:parts} we get the existence of $D_{0+}^\alpha y_{c,H}$ and $D_{T-}^\alpha y_{c,H}.$ Consequently, the statement of the theorem holds.
\end{proof}

\section{The
Riemann-Stieltjes integral in terms of weighted Takagi-Landsberg functions}
Let $\alpha\in (0,1).$ Denote by $\mathbf{H}^\alpha[0,1]$  the space of $\alpha-$H\"older continuous function on $[0,1].$ In this section we consider the
Riemann-Stieltjes integral of $f\in \mathbf{H}^{H_1}[0,1]$ with respect to $g\in \mathbf{H}^{H_2}[0,1]$ if $H_1+H_2>1,$ which   can be defined  as
\begin{equation*}
    \int_0^t f dg=-\int_{0}^t D_{0+}^\alpha f(s)D_{t-}^{1-\alpha}[g(\cdot) - g (t)](s)ds
\end{equation*}
for any $\alpha\in (0,1)$ such that $\alpha < H_1, 1 -\alpha < H_2,$ see, e.g. \cite{Zahle}.

We use the Takagi-Landsberg representation of functions $f$ and $g$ \eqref{ydef} to give the series expansion of integral $\int_0^t f dg.$
Denote by
\begin{equation}
    \label{D:def}
    \Delta^\alpha_{2^m+k,2^n+l}(t)=\tau^{\alpha}_{1,2^m+k}\left(t\wedge\frac{l}{2^n}\right)-2\tau^{\alpha}_{1,2^m+k}\left(t\wedge \frac{l+0.5}{2^n}\right)+\tau^{\alpha}_{1,2^m+k}\left(t\wedge \frac{l+1}{2^n}\right) \, t\in[0,1]
\end{equation}
for $\alpha>0,$ $n,m\in\N_0,$ $l=0,\ldots,2^n-1,k=0,\ldots 2^m-1.$
\begin{theorem}
\label{DI} Let $f\in \mathbf{H}^{H_1}[0,1]$ and $g\in \mathbf{H}^{H_2}[0,1]$ with $H_1+H_2>1$ possess  the following Takagi-Landsberg representations
\begin{align*}
    f(t)&=\sum_{m=0}^\infty 2^{m\left(\frac{1}{2}-H_1\right)}\sum_{k=0}^{2^m-1}{c^{(1)}_{m,k}}e_{m,k}(t), t\in [0,1],\\
    g(t)&=\sum_{m=0}^\infty 2^{m\left(\frac{1}{2}-H_2\right)}\sum_{k=0}^{2^m-1}{c^{(2)}_{m,k}}e_{m,k}(t), t\in [0,1],
\end{align*}
where $|c^{(1)}_{m,k}|,|c^{(2)}_{m,k}|\leq L$ for some $L>0.$
If $1-H_2<\alpha<H_1,$ then
\begin{align}
\nonumber    &\int_0^t f(s)dg(s)=-\int_{0}^t D_{0+}^\alpha f(s)D_{t-}^{1-\alpha}[g(\cdot) - g (t)](s)ds\\
\label{DI:eq0}    &=-\sum_{n=0}^\infty \sum_{m=0}^\infty \sum_{k=0}^{2^m-1} \sum_{l=0}^{2^n-1}2^{m(1-H_1)+n(1-H_2)}c^{(1)}_{m,k} c^{(2)}_{n,l}\Delta^2_{2^m+k,2^n+l}(t).
\end{align}
\end{theorem}
\begin{proof}

Due to Theorem \ref{thmWF}, we have that $D_{0+}^{\alpha} f$ and $D_{t-}^{1-\alpha} (g(\cdot)-g(t))$ exist and converge uniformly as series \eqref{D0y} and \eqref{DTy}.
Therefore, $D_{0+}^\alpha f(s)D_{t-}^{1-\alpha}[g(\cdot) - g (t)](s)$ converges uniformly on $s\in (0,t)$ as well with the following bound
\begin{align*}\left| D_{0+}^\alpha f(s)D_{t-}^{1-\alpha}[g(\cdot) - g (t)](s)\right|
&\leq \sum_{n=0}^\infty \sum_{m=0}^\infty 2^{m(\alpha-H_1)+n(1-\alpha-H_2)} L^2 c_1(\alpha)c_1(1-\alpha),
\end{align*}
for all $s\in (0,t).$ So, we apply the Lebesgue dominated convergence theorem to the integral $\int_{0}^t D_{0+}^\alpha f(s)D_{t-}^{1-\alpha}[g(\cdot) - g (t)]ds,$ which  equals now
\begin{align}\label{integ}
    \sum_{n=0}^\infty \sum_{m=0}^\infty \sum_{k=0}^{2^m-1} \sum_{l=0}^{2^n-1}2^{m(\frac{1}{2}-H_1)+n(\frac{1}{2}-H_2)}c^{(1)}_{m,k} c^{(2)}_{n,l}\int_{0}^tD_{0+}^\alpha e_{m,k}(s)D_{t-}^{1-\alpha} [e_{n,l}-e_{n,l}(t)](s)ds.
\end{align}

Compute the   integral in \eqref{integ} using Proposition \ref{propemk}:
\begin{align}
\nonumber&\int_{0}^tD_{0+}^\alpha e_{m,k}(s)D_{t-}^{1-\alpha}[ e_{n,l}(\cdot)-e_{n,l}(t)](s)ds=-\int_{0}^t I_{0+}^{1-\alpha} H_{m,k}(s)I_{t-}^{\alpha} H_{n,l}(s)ds.\\
\nonumber&=-\frac{1}{\Gamma(1-\alpha)}\frac{1}{\Gamma(\alpha)} \int_{0}^t \int_0^s\int_s^t (s-u_1)^{-\alpha}H_{m,k}(u_1)(u_2-s)^{\alpha-1}H_{n,l}(u_2) du_2 du_1 ds \\
\nonumber&=-\frac{1}{B(1-\alpha,\alpha)}\int_{0}^t H_{n,l}(u_2)\int_0^{u_2}H_{m,k}(u_1)\int_{u_1}^{u_2} (s-u_1)^{-\alpha}(u_2-s)^{\alpha-1}ds du_1 du_2 \\
\label{DI:eq3}&=-\int_{0}^t H_{n,l}(u_2)\int_0^{u_2}H_{m,k}(u_1) du_1 du_2=-\int_{0}^t H_{n,l}(u)e_{m,k}(u)du.
\end{align}
Obviously, if $t<\frac{k}{2^m}\vee \frac{l}{2^m}$ the last integral equals zero.
\\
Let $t \in J_{n,l},$ then
\begin{align*}
  &\int_{0}^t H_{n,l}(u)e_{m,k}(u)du=   2^{\frac{n}{2}}\int_{\frac{l}{2^n}}^{t} e_{k,m}(u)du=2^{\frac{n}{2}}\left(I_{0+}^2H_{m,k}(t)-I_{0+}^2H_{m,k}\left(\frac{l}{2^n}\right)\right).
\end{align*}
If $t \in J_{n,l+0.5},$ then
\begin{align*}
  &\int_{0}^t H_{n,l}(u)e_{m,k}(u)du=2^{\frac{n}2}\int_{\frac{l}{2^n}}^{\frac{l+0.5}{2^n}}e_{m,k}(u)du-\int_{\frac{l+0.5}{2^n}}^t e_{m,k}(u)du\\
  &= 2^{\frac{n}{2}}\left(2I_{0+}^2H_{m,k}\left(\frac{l+0.5}{2^n}\right)-I_{0+}^2H_{m,k}\left(\frac{l}{2^n}\right)-I_{0+}^2H_{m,k}\left(t\right)\right).
\end{align*}
The case $t> \frac{l+1}{2^n}$ is similar. Thus, we have
\begin{align*}
  &\int_{0}^t H_{n,l}(u)e_{m,k}(u)du\\
  &=2^{\frac{n}{2}}\left(2I_{0+}^2H_{m,k}\left(t\wedge \frac{l+0.5}{2^n}\right)-I_{0+}^2H_{m,k}\left(t\wedge\frac{l}{2^n}\right)-I_{0+}^2H_{m,k}\left(t\wedge \frac{l+1}{2^n}\right)\right).
\end{align*}

Note that
\begin{align*}
    &\int_{0}^t H_{n,l}(u)e_{m,k}(u)du=\int_{0}^t H_{n,l}(u_2)\int_0^{u_2}H_{m,k}(u_1) du_1 du_2\\
    &=\int_{0}^t H_{m,k}(u_1) \int_{u_1}^{t}H_{n,l}(u_2)du_2 du_1 =\int_{0}^t H_{m,k}(u) [e_{n,l}(t)-e_{n,l}(u)]du.
\end{align*}
Then the statement follows from Lemma \ref{haarintegr}, relations \eqref{taudef} and \eqref{D:def}.
\end{proof}

\begin{remark}
The Riemann-Stieltjes integral in Theorem \ref{DI} can be written as
\begin{align}
 \nonumber   &\int_{0}^t f(s)dg(s)=-\int_{0}^t D_{0+}^\alpha f(s)D_{t-}^{1-\alpha}[g(\cdot) - g (t)](s)ds\\
\nonumber   &=\sum_{n=0}^\infty \sum_{l=0}^{2^n-1}2^{n\left(\frac{1}{2}-H_2\right)} c^{(2)}_{n,l}\int_{0}^t H_{n,l}(u)f(u)du\\
\nonumber    &=\sum_{m=0}^\infty \sum_{k=0}^{2^m-1} 2^{m\left(\frac{1}{2}-H_1\right)}c^{(1)}_{m,k} \int_{0}^t H_{m,k}(u) [g(t)-g(u)]du\\
\label{form3}&=\sum_{n=0}^\infty \sum_{m=0}^\infty \sum_{l_1=0}^{2^{n_1}-1} \sum_{l_2=0}^{2^{n_2}-1}2^{n_1(\frac{1}{2}-H_1)+n_2(\frac{1}{2}-H_2)}c^{(1)}_{n_1,l_1} c^{(2)}_{n_2,l_2} \int_{0}^t H_{n_2,l_2}(u)e_{n_1,l_1}(u)du.
\end{align}
\end{remark}

\begin{remark} Particularly, we have
\begin{align*}
    \int_0^t f(s) ds &=-\int_{0}^t D_{0+}^\alpha f(s)D_{t-}^{1-\alpha}[(\cdot) -  t](s)ds=I^1_{0+}f (t),\\
    \int_0^t dg(s) &=-\int_{0}^t D_{0+}^\alpha \mathbbm{1}_{[0,t]}(s)D_{t-}^{1-\alpha}[g(\cdot) -  g(t)](s)ds=g(t)-g(0),\\
    \int_0^t s dg(s) &=-\int_{0}^t D_{0+}^\alpha[(\cdot)](s)D_{t-}^{1-\alpha}[g(\cdot) -  g(t)](s)ds=tg(t)-I^1_{0+}g (t).
\end{align*}
\end{remark}

From \cite[Proposition 4.4.1]{Zahle} it follows that $\int_0^{\cdot}fdg \in \mathbf{H}^{H_2}[0,1].$

\begin{corollary} \label{cor1} The coefficients $x_0^R,x_1^R, c^R$ in Takagi-Landsberg representation of the Riemann-Stieltjes integral in Theorem \ref{DI} equal $x^R_0=0,$
\begin{equation*}
    x^R_1=-\sum_{n_1=0}^\infty \sum_{n_2=0}^\infty \sum_{l_1=0}^{2^{n_1}-1} \sum_{l_2=0}^{2^{n_2}-1}2^{n_1(1-H_1)+n_2(1-H_2)}c^{(1)}_{n_1,l_1} c^{(2)}_{n_2,l_2} \Delta^2_{2^{n_1}+l_1,2^{n_2}+{l_2}}(1),
\end{equation*}
\begin{align*}
    c^R_{m,k}&=\sum_{n_1=0}^\infty \sum_{l_1=0}^{2^{n_1}-1}c^{(1)}_{n_1,l_1} \sum_{n_2=0}^\infty 2^{(n_1+n_2-m)(\frac{1}{2}-H)} \sum_{l_2=0}^{2^{n_2}-1}c^{(2)}_{n_2,l_2}\int_{0}^{1} e_{n_1,l_1}(u)H_{n_2,l_2}(u)H_{m,k}(u)du\\
    &=2^{mH_2}\sum_{n_1=0}^\infty \sum_{l_1=0}^{2^{n_1}-1}c_{n_1,l_1} \sum_{n_2=0}^\infty \frac{2^{n_1(\frac{1}{2}-H_1)}}{2^{(m-n_2)(\frac{1}{2}-H_2)}} \sum_{l_2=0}^{2^{n_2}-1}c^{(2)}_{n_2,l_2}
    \\
    &\times\left(\Delta^2_{2^{n_1}+l_1,2^{n_2}+{l_2}}\left(\frac{k}{2^m}\right)-2\Delta^2_{2^{n_1}+l_1,2^{n_2}+{l_2}}\left(\frac{k+0.5}{2^m}\right)+\Delta^2_{2^{n_1}+l_1,2^{n_2}+{l_2}}\left(\frac{k+1}{2^m}\right)\right).
\end{align*}
\end{corollary}
\begin{proof}
The value of $x_1^R$ follows from \eqref{DI:eq0}. Denote by $R(t)$ the value of the integral $\int_0^t f dg.$  The function $R\in\mathbf{H}^{H_2}[0,1]$ possesses the representation as a weighted Takagi-Landsberg function with coefficients $c^R$ given by
    $c^R_{m,k}=2^{mH}\left[2R\left(\frac{k+0.5}{2^m}\right)-R\left(\frac{k+1}{2^m}\right)-R\left(\frac{k}{2^m}\right)\right].$
Then for $m\in \N_0$ and $k=0,\ldots,2^m-1$ we have from \eqref{form3} that
\begin{align*}
    &c^R_{m,k}=2^{mH_2}\sum_{n_1=0}^\infty \sum_{l_1=0}^{2^{n_1}-1}c^{(1)}_{n_1,l_1} \sum_{n_2=0}^\infty 2^{n_1(\frac{1}{2}-H_1)+n_2(\frac{1}{2}-H_2)} \sum_{l_2=0}^{2^{n_2}-1}c^{(2)}_{n_2,l_2}\\
    &\times\left(2\int_{0}^{\frac{k+0.5}{2^m}} e_{n_1,l_1}(u)H_{n_2,l_2}(u)du-\int_{0}^{\frac{k}{2^m}} e_{n_1,l_1}(u)H_{n_2,l_2}(u)du-\int_{0}^{\frac{k+1}{2^m}} e_{n_1,l_1}(u)H_{n_2,l_2}(u)du\right)\\
    &=\sum_{n_1=0}^\infty \sum_{l_1=0}^{2^{n_1}-1}c_{n_1,l_1} \sum_{n_2=0}^\infty \frac{2^{n_1(\frac{1}{2}-H_1)}}{2^{(m-n_2)(\frac{1}{2}-H_2)}} \sum_{l_2=0}^{2^{n_2}-1}c^{g}_{n_2,l_2}\int_{0}^{1} e_{n_1,l_1}(u)H_{n_2,l_2}(u)H_{m,k}(u)du.
\end{align*}

We can rewrite the last integral as
\begin{align}
    &\int_{0}^{1} e_{n_1,l_1}(u)H_{n_2,l_2}(u)H_{m,k}(u)du=2^{\frac{n_1+n_2+m}{2}}\\
\nonumber    &\times\left(\Delta^2_{2^{n_1}+l_1,2^{n_2}+{l_2}}\left(\frac{k}{2^m}\right)-2\Delta^2_{2^{n_1}+l_1,2^{n_2}+{l_2}}\left(\frac{k+0.5}{2^m}\right)+\Delta^2_{2^{n_1}+l_1,2^{n_2}+{l_2}}\left(\frac{k+1}{2^m}\right)\right).
\end{align}
\end{proof}

\begin{remark}
\label{remark:sg}Let $g(0)=g(1)=0.$ The integral $\int_0^t s dg(s)$ possesses  the following Takagi- Landsberg representation
\begin{align}
\nonumber    &\int_0^t s dg(s)=tg(t)-I^1_{0+}g(t)=
    -t \sum_{n=0}^\infty \sum_{l=0}^{2^{n}-1}2^{-{n}(1+H_2)-2}c^{(2)}_{n,l}\\
\label{int:sg}    &+\sum_{m=0}^\infty \sum_{k=0}^{2^{m}-1}2^{{m}\left(\frac{1}{2}-H_2\right)}e_{m,k}(t)\left[\frac{k c_{m,k}^{(2)}}{2^{m}}+2^{mH_2}\sum_{n=0}^\infty \sum_{l=0}^{2^{n}-1}2^{{n}\left(\frac{1}{2}-H_2\right)}c^{(2)}_{n,l}D_{2^n+l,2^m+k}\right].
\end{align}
where
$D_{2^n+l,2^m+k}:=\frac{1}{2^{m }}e_{n,l}\left(\frac{k+0.5}{2^m}\right)-\frac{1}{2^{m }}e_{n,l}\left(\frac{k+1}{2^m}\right)+2^{\frac{n}{2}}\Delta^2_{2^n+l,2^m+k}(1).$
\end{remark}

\section{Applications to fractional integral equations}
In this section, we solve integral equations, involving fractional integrals and derivatives, with the help of the  Takagi-Landsberg representations of the H\"{o}lder continuous functions.  To do so, we use the uniqueness of the Schauder expansion.

Let $H\in(0,1)$ and $g\in \mathbf{H}^H[0,1]$ have the  Takagi-Landsberg representation \eqref{TLr} with coefficients $g_0,g_1,c^g=\{c^g_{m,k}\}.$

Denote by $S_m$ the operator that gives the partial sums of the Takagi-Landsberg expansion of $g$ by
\begin{equation}
    [S_m g](t):=g_0(1-t)+g_1 t+ \sum_{n=0}^{m}\sum_{l=0}^{2^n-1}2^{n\left(\frac{1}{2}-H\right)}c^g_{n,l}e_{n,l}(t),t\in [0,1].
\end{equation}
From the properties of the Schauder system  we get that  $g\left(\frac{k}{2^m}\right)=[S_{m-1} g]\left(\frac{k}{2^m}\right),$  $0\leq k\leq 2^m-1.$

In this section it is also convenient to make the new indexation of $c^g.$ We write $c^g_{n}$ for $c^g_{m,k}$ if $n=2^m+k,$ $m\geq 0, k=0,\ldots,2^m-1.$

\begin{remark}
Let $X:[0,1]\to \R$ be a $H_1$-H\"{o}lder continuous function. Consider a function $f:\R\to\R$ such  that  $f(X)\in \mathbf{H}^{H_2}.$  If $X$ admits representation \eqref{ydef} with coefficients $c_{m,k}^x,$ then $f(X)$ has representation  with coefficients $c_{m,k}^f,$ where
\begin{equation}
\label{coef:def}c_{m,k}^x=2^{mH_1}\left[2X\left(\frac{k+0.5}{2^m}\right)-X\left(\frac{k+1}{2^m}\right)-X\left(\frac{k}{2^m}\right)\right]
\end{equation}
and
\begin{align*}
    c_{m,k}^f&=2^{mH_2}\left[2f\left(X\left(\frac{k+0.5}{2^m}\right)\right)-f\left(X\left(\frac{k+1}{2^m}\right)\right)-f\left(X\left(\frac{k}{2^m}\right)\right)\right]\\
    &=2^{mH_2}\left[2f\left(S_{m}X\left(\frac{k+0.5}{2^m}\right)\right)-f\left(S_{m-1}X\left(\frac{k+1}{2^m}\right)\right)-f\left(S_{m-1}X\left(\frac{k}{2^m}\right)\right)\right].
\end{align*}
Thus, coefficients $c^f_m$ are determined by coefficients $\{c^x_n,n\leq m\}.$
\end{remark}

\subsection{Volterra integral equation}

Let $H<\alpha\in (0,1),$  $\theta\neq 0$ and $g\in \mathbf{H}^H[0,1],$ that is $g$ has the  Takagi-Landsberg representation with bounded coefficients $c^g=\{c_{m,k}^g\}$.   Consider the Volterra integral equation given by
\begin{equation}
\label{Lang:eq:def}    X(t)=x_0+\theta [I^{\alpha}X](t)+g(t), t\in [0,1].
\end{equation}
Equation \eqref{Lang:eq:def} is called also as the fractional Langevin equation, e.g. \cite{Fa}.

It follows from the general theory of integral equations  that  \eqref{Lang:eq:def} has a unique solution in $C[0,1]$, e.g. \cite[Section XII.6.2]{kantor}. Indeed, the operator $I^{\alpha}_{0+}$ has the norm $\|I^{\alpha}_{0+}\|_\infty=\frac{1}{\Gamma(\alpha)}\max_{t\in [0,1]}\left(\int_0^t (t-s)^{\alpha-1}ds\right)=\frac{1}{\Gamma(1+\alpha)}.$ Moreover, by \cite[formula (2.21)]{skm} its powers equal $[I^{\alpha}_{0+}]^n=I^{\alpha n}_{0+}$ with $\left\|[I^{\alpha}_{0+}]^n\right\|_\infty=\frac{1}{\Gamma(\alpha n+1)}.$ Denote by $\tilde{g}=x_0+g.$ A solution $X$ of equation \eqref{Lang:eq:def} can be expanded as a power series $X=\tilde{g} + \theta I^{\alpha}_{0+}\tilde{g}+\ldots +\theta^n I^{n\alpha}_{0+}\tilde{g}+\ldots$ which converges for all $\theta$ with
$$
    |\theta|< \lim_{n\to \infty} \left\|[I^{\alpha}_{0+}]^n\right\|^{-\frac{1}{n}}_\infty=\lim_{n\to \infty} \left(\Gamma(\alpha n +1)\right)^{\frac{1}{n}}=\lim_{n\to \infty} (2\pi)^{\frac{1}{2n}}e^{-\alpha} (n\alpha )^{\alpha-\frac{1}{2n}}=\infty,
$$
where the asymptotic behavior  of the Gamma function is given by \cite[formula 6.1.39]{abram}. Since operator $I^{\alpha}_{0+}$ maps $C[0,1]$ into $H^\alpha[0,1]$ (e.g. \cite[Corrolary 2, p 58]{skm}), the solution of \eqref{Lang:eq:def} belongs to $H^H[0,1].$

Thus, $X$ posses the Takagi-Landsberg representation \eqref{TLr} with $x_1\in \R$ and bounded coefficients $c^x=\{c_m^x,m\geq 0\}$
$$X(t)=x_0+(x_1-x_0)t+\sum_{n=0}^\infty\sum_{l=0}^{2^n-1}2^{n\left(\frac{1}{2}-H\right)}c^x_{2^n+l}e_{n,l}(t), t\in[0,1].$$
Then we  apply Lemma \ref{diff:parts} and formula \eqref{D0y} to get that $[I^{\alpha}X](t)$ has the following series representation
\begin{align*}
    &[I^{\alpha}X](t)=\frac{x_0}{\Gamma(1+\alpha)}t^\alpha+\frac{x_1-x_0}{\Gamma(\alpha+2)}t^{1+\alpha}+\sum_{n=0}^\infty\sum_{l=0}^{2^n-1}2^{n\left(1-H\right)}c^x_{2^n+l}\tau^{1+\alpha}_{1,2^n+l}(t), t\in[0,1].
\end{align*}

We introduce a truncated fractional integral $I^{\alpha}_{0+}S_p:\mathbf{H}^\alpha[0,1]\to \mathbf{H}^\alpha[0,1],$ of order $p\in \N$ as
\begin{align*}&[I^{\alpha}_{0+} S_p X](t)=\frac{x_0}{\Gamma(1+\alpha)}t^\alpha+\frac{x_1-x_0}{\Gamma(\alpha+2)}t^{1+\alpha}+\sum_{n=0}^p\sum_{l=0}^{2^n-1}c^x_{2^n+l}2^{n(1-H)}\tau^{1+\alpha}_{1,2^n+l}(t), t\in[0,1].
\end{align*}

Denote by $X_p$ the solution of the following truncated equation
\begin{equation}
\label{Lang:eq2}    X_p(t)=x_0+\theta [I^{\alpha}S_p X_p](t)+g(t), t\in [0,1].
\end{equation}
Obviously, $\|I^{\alpha}_{0+} S_p\|_{\infty}\leq \|I^{\alpha}_{0+}\|_\infty,$ thus \eqref{Lang:eq2} has a unique solution in $C[0,1].$ By construction $I^{\alpha}_{0+} S_p X_p\in \mathbf{H}^\alpha[0,1],$ so $X_p$ is  $H-$H\"{o}lder continuous on $[0,1]$ as well.

Here we give the solution of \eqref{Lang:eq2} by finding the coefficients $c^p$ and $x_1^p$ in the Takagi-Landsberg expansion \eqref{TLr} of $X_p$.

Denote by
\begin{align}
\label{a11}    &a_{2^m+k,2^n+l}=-\theta 2^{mH+n(1-H)}\Delta^{1+\alpha}_{2^n+l,2^m+k}\left(1\right),\quad a_{2^m+k,0}=2^{mH}\theta \tau^{1+\alpha}_{2,2^m+k}(0,1), \\
\label{a12}    & b_{2^m+k}=c^g_{2^m+k}+2^{mH}\theta x_0\tau^\alpha_{2,2^m+k}(0,1)-2^{mH}\theta x_0\tau^{1+\alpha}_{2,2^m+k}(0,1).
\end{align}
\begin{equation}
    \label{a13} b_0=x_0+g_{1}+\frac{\theta x_0}{\Gamma(1+\alpha)}-\frac{\theta x_0}{\Gamma(\alpha+2)}, \, a_{0,0}=\frac{\theta }{\Gamma(\alpha+2)}, \, a^p_{0,2^n+l}=\theta2^{n(1-H)}\tau^{1+\alpha}_{1,2^n+l}(1).
\end{equation}
\begin{lemma}
\label{Lang:sol}
Let $p\geq 1,P=2^{p+1}-1$ and denote by   $A_p=(a_{k,l})_{k,l=0}^{P},$
$\mathbf{C}_p=(x^p_1,c^p_1,\ldots,c^p_{P})^T,$ and $\mathbf{b}_p=(b_0,\ldots,b_P)^T,$ where $a_{k,l}$ and $b_k$ are given by \eqref{a11}-\eqref{a13}.
Let $\mathbf{C}_p$ be a solution of
\begin{equation}
    \label{Lang:eq5} \mathbf{C}_p=A_p \mathbf{C}_p +\mathbf{b}_p
\end{equation}
and $c^p_{m}=b_m+x_1^p a_{m,0}+\sum_{n=1}^{P}c^p_n a_{m,n},$ $m> P.$
Then the function
$$X_p=x_0(1-t)+x_1^p t +\sum_{m=0}^\infty\sum_{k=0}^{2^m-1}2^{m(\frac12-H)}c^p_{2^m+k}e_{m,k}(t), t\in [0,1]$$
is the solution of equation \eqref{Lang:eq2}.
\end{lemma}
\begin{proof}
Since the Takagi-Lansdberg expansion is unique, and its coefficients are determined by \eqref{coef:def},   we have the following relations
\begin{align}
\nonumber     &c^p_{2^m+k}=c^g_{2^m+k}+2^{mH}\theta x_0\tau^\alpha_{2,2^m+k}(0,1)+2^{mH}\theta (x_1^p-x_0)\tau^{1+\alpha}_{2,2^m+k}(0,1)\\
\nonumber     &+2^{mH}\theta \sum_{n=0}^p\sum_{l=0}^{2^n-1}c^p_{2^n+l}2^{n(1-H)}\left(2\tau^{1+\alpha}_{1,2^n+l}\left(\frac{k+0.5}{2^m}\right)-\tau^{1+\alpha}_{1,2^n+l}\left(\frac{k+1}{2^m}\right)-\tau^{1+\alpha}_{1,2^n+l}\left(\frac{k}{2^m}\right)\right)\\
\label{Lang:eq3}    &=b_{2^m+k}+x_1^pa_{2^m+k,0}+\sum_{n=0}^p\sum_{l=0}^{2^n-1}c^p_{2^n+l}a^p_{2^m+k,2^n+l}.
\end{align}

At point $t=1$ equation \eqref{Lang:eq2} gives the next relation
\begin{align}
\nonumber    x_1^p&=x_0+g_{1}+\frac{\theta x_0}{\Gamma(1+\alpha)}+\frac{\theta (x_1^p-x_0)}{\Gamma(\alpha+2)}+\theta \sum_{n=0}^p\sum_{l=0}^{2^n-1}c^p_{2^n+l}2^{n(1-H)}\tau^{1+\alpha}_{1,2^n+l}(1)\\
\label{Lang:eq4}    &=b_0+x_1^pa_{0,0}+\sum_{n=0}^p\sum_{l=0}^{2^n-1}c^p_{2^n+l}a^p_{0,2^n+l}.
\end{align}

Then relations \eqref{Lang:eq3} and \eqref{Lang:eq4} yield the statement of the Lemma.
\end{proof}

\begin{lemma}
\label{lemma:cont1}
Let $X_p$ be the solution of equation \eqref{Lang:eq2}, then $X_p$ tends to the solution of \eqref{Lang:eq:def} in the supremum norm on [0,1].
\end{lemma}
\begin{proof}
Let $X$ be the solution  of \eqref{Lang:eq:def}.
Denote by $err_p=X_p-X.$ 
Note that \begin{equation}
    \label{Lang:eq6} err_p(t)=\theta I^{\alpha}_{0+}S_p X_p(t) - \theta  I^{\alpha}_{0+}X(t)= \theta I^{\alpha}_{0+}err_p(t) + \theta I^{\alpha}_{0+}[S_p X_p-X_p](t).
\end{equation}
Due to the power series expansion of  $err_p$ as a solution of equation \eqref{Lang:eq6} we have
$$\|err_p\|_{\infty}\leq \left(1+\sum_{n=1}^\infty |\theta|^n \|I^{\alpha n}_{0+}\| \right) \left\|\theta I^{\alpha}_{0+}[S_p X_p-X_p]\right\|_\infty.$$

Then $|X_p(t)-S_p X_p(t)|\leq\sum_{m=p+1}^{\infty}\sum_{k=0}^{2^m-1}2^{m\left(\frac{1}{2}-H\right)}|c_{m,k}|e_{m,k}(t)\leq L (x^H(t)-S_p x^H(t)),$ where $x^H$ is a Takagi-Landsberg function. The second term in the right-hand side of \eqref{Lang:eq6} is bounded by
\begin{align}
\label{estSP1}&\frac{|\theta|}{\Gamma(\alpha)}\left|\int_0^t\frac{S_p X_p (u)-X(u)}{(t-u)^{1-\alpha}}du\right|\leq \frac{1}{\Gamma(\alpha)}\int_0^t\frac{\left|S_p X_p (u)-X_p(u)\right|}{(t-u)^{1-\alpha}}du\\
\nonumber&\leq  \frac{L}{\Gamma(\alpha)}\int_0^t\frac{x^H(u)-S_p x^H (u)}{(t-u)^{1-\alpha}}du\leq \frac{Lt^{\alpha}}{\Gamma(1+\alpha)}\sup_{u\in[0,1]}(x^H(u)-S_p x^H (u)).
\end{align}
Thus, $\|I^{\alpha}_{0+}[S_p X_p-X_p](t)\|_{\infty}\to 0$ as $p\to \infty.$ This yields that $\left\|X-X_p\right\|_\infty\to 0,p\to \infty.$
\end{proof}

\subsection{A linear differential equation}

Let $\beta,\gamma\in \R$ and $\beta\neq 0,$ $\gamma\neq 0.$ Let $g:[0,1]\to \R$ be a H\"{o}lder continuous of order $H>\frac{1}{2},$ with $g(0)=g(1)=0,$ that is $g$ be a weighted Takagi-Landsberg function with bounded coefficients $c^g=\{c_{m,k}^g,m\geq 0, k=0,\ldots,2^m-1\}.$ Let $\alpha \in (1-H,H).$ Consider the linear equation
\begin{align}
 \nonumber   X(t)&=x_0+\beta I_{0+}^1X (t)+\gamma\int_{0}^t X(z)d g(z)\\
\label{lin:eq:def}&=x_0+\beta\int_0^t X(z)dz-\gamma  \int_{0}^t [D_{0+}^\alpha X](z) [D_{t-}^{1-\alpha}(g(\cdot)-g(t))](z)dz, t\in [0,1].
\end{align}

Denote by $U: \mathbf{H}^H\to \mathbf{H}^H$ the operator $U(x)=\beta I_{0+}^1x +\gamma \int_{0}^\cdot x d g.$
It was shown in \cite{nualart} that $U$ is a compact linear operator on Banach space  $W_0^{\alpha,\infty}$ with respect to the the norm $\|f\|_{\alpha,\infty}:=\sup_{t\in[0,1]}\left(|f(t)|+\int_0^t\frac{|f(t)-f(s)|}{(t-s)^{\alpha+1}}ds\right)$ and for $\lambda\geq 0$ an equivalent norm is defined by $\|f\|_{\alpha,\lambda}:=\sup_{t\in[0,1]}e^{-\lambda t}\left(|f(t)|+\int_0^t\frac{|f(t)-f(s)|}{(t-s)^{\alpha+1}}ds\right).$ Moreover, there exists $\lambda_0>0$ such that
$$\|U(x)-U(y)\|_{\alpha,\lambda_0}\leq \frac{1}{2}\|x-x\|_{\alpha,\lambda_0},  \forall x,y\in U(B_0)\subset B_0=\{u\in W_0^{\alpha,\infty}: \|u\|_{\alpha,\lambda_0}\leq 2(1+|x_0|)\}.$$
This ensures, that there exists a unique solution $X\in W_0^{\alpha,\infty}$ of equation \eqref{lin:eq:def}, e.g. \cite[Theorem 5.1]{nualart}.

Let us apply the Takagi-Landsberg expansion to solve \eqref{lin:eq:def}.
Using notation \eqref{taudef}, we get that the first integral in the right hand side of \eqref{lin:eq:def} has the following representation
$$\int_0^t X(s)ds=x_0 t+\frac{x_1-x_0}{2}t^{2}-\sum_{n=0}^\infty\sum_{l=0}^{2^n-1}2^{n(1-H)}c^x_{2^n+l}\tau^2_{1,2^n+l}(t), t\in[0,1].
$$
The Riemann-Stieltjes integral in \eqref{lin:eq:def}  is $H-$H\"{o}lder continuous and admits the following representation due to Theorem \ref{DI} and Remark \ref{remark:sg}.
\begin{align}&\int_0^t X(s)dg(s)=x_0 g(t)+(x_1-x_0)(tg(t)-I_{0+}^1 g(t))\\
&- \sum_{n=0}^\infty \sum_{m=0}^\infty \sum_{k=0}^{2^m-1} \sum_{l=0}^{2^n-1}2^{m(1-H_1)+n(1-H_2)}c^{x}_{m,k} c^{g}_{n,l}\Delta^2_{2^m+k,2^n+l}(t), t\in[0,1].
\end{align}

Denote by $X_p$ the solution of the following truncated equation
\begin{equation}
\label{LE:eq2}    X_p(t)=x_0+\beta [I_{0+}^1 S_p X_p](t) -\gamma  \int_{0}^t [D_{0+}^\alpha S_p X](z) [D_{t-}^{1-\alpha} S_p (g- g(t))](z)dz, t\in [0,1].
\end{equation}

Denote by
\begin{align}
\label{a21}  & a_{2^m+k,2^n+l}=-2^{mH}\beta 2^{n(1-H)} \Delta^2_{2^n+l,2^m+k}(1)+\gamma \sum_{n_2=0}^p 2^{(n+n_2)(1-H)+mH} \sum_{l_2=0}^{2^{n_2}-1}c^{g}_{n_2,l_2}\\
&\times\left(\Delta^2_{2^{n}+l,2^{n_2}+{l_2}}\left(\frac{k}{2^m}\right)-2\Delta^2_{2^{n}+l,2^{n_2}+{l_2}}\left(\frac{k+0.5}{2^m}\right)+\Delta^2_{2^{n}+l,2^{n_2}+{l_2}}\left(\frac{k+1}{2^m}\right)\right),\\
\label{a22}    &a_{2^m+k,0}=-\frac{\beta x_0 2^{m(H-2)}}{4}+\gamma\frac{k c^{g}_{m,k}}{2^m}+\gamma\sum_{n_2=0}^p \sum_{l_2=0}^{2^{n_2}-1}2^{mH}2^{{n_2}\left(\frac{1}{2}-H\right)}c^g_{n_2,l_2}D_{2^n_2+l_2,2^m+k},\\
\label{a23}    &a_{0,2^n+l}=\beta 2^{-n(H+1)-2}-\gamma
   \sum_{n_2=0}^p 2^{(n_1+n_2)(1-H)} \sum_{l_2=0}^{2^{n_2}-1}c^{g}_{n_2,l_2} \Delta_{2^n+l,2^{n_2}+l_2}(1),\\
\label{a24}  &b_{2^m+k}=-a_{2^m+k,0}+\gamma x_0c_{m,k}^{g}, \,  a_{0,0}=\frac{\beta}{2}-\gamma \sum_{n_2=0}^p \sum_{l_2=0}^{2^{n_2}-1}2^{-{n_2}(1+H)-2}c^g_{n_2,l_2},\\
\label{a25}&b_{0}=\frac{\beta x_0}{2}+\gamma x_0\sum_{n_2=0}^p \sum_{l_2=0}^{2^{n_2}-1}2^{-{n_2}(1+H)-1}c^g_{n_2,l_2}.
\end{align}

\begin{lemma}
\label{LE:sol}
Let $p\geq 1,P=2^{p+1}-1$ and denote by   $A_p=(a_{k,l})_{k,l=0}^{P},$
$\mathbf{C}_p=(x^p_1,c_1^p,\ldots,c^p_{P})^T,$ and $\mathbf{b}_p=(b_0,\ldots,b_P)^T,$ where $a_{k,l}$ and $b_k$ are given by \eqref{a21}-\eqref{a25}.
Let $\mathbf{C}_p$ be a solution of
\begin{equation}
    \label{LE:eq5} \mathbf{C}_p=A_p \mathbf{C}_p +\mathbf{b}_p
\end{equation}
and $c^p_{m}=b_m+x_1^p a_{m,0}+\sum_{n=1}^{P}c^p_n a_{m,n},$ $m> P$.
Then the function
$$X_p=x_0(1-t)+x_1^p t +\sum_{m=0}^\infty\sum_{k=0}^{2^m-1}2^{m(\frac12-H)}c^p_{2^m+k}e_{m,k}(t), t\in [0,1]$$
is the solution of equation \eqref{LE:eq2}.
\end{lemma}

\begin{proof}From Remark \ref{remark:sg}, Lemma \ref{Lang:sol} and Corollary \ref{cor1} we have the following relations for the coefficients $c^p$ in the Takagi-Landsberg expansion of $X_p.$
\begin{align*}
\nonumber     c^p_{2^m+k}&=\beta (x_0-x_1^p)2^{m(H-2)-2}-2^{mH}\beta \sum_{n=0}^p\sum_{l_1=0}^{2^{n_1}-1}c^p_{2^{n_1}+l_1}2^{n_1(1-H)}\Delta^2_{2^{n_1}+l_1,2^{m}+k}(1)+\gamma x_0c_{m,k}^g\\
\nonumber &+\gamma(x^p_1-x_0)\frac{k c_{m,k}^{g}}{2^{m}}+\gamma(x^p_1-x_0)\sum_{n_2=0}^p \sum_{l_2=0}^{2^{n_2}-1}2^{mH}2^{{n_2}\left(\frac12-H\right)}c^g_{n_2,l_2}D_{2^{n_2}_2+l_2,2^m+k}\\
\nonumber &+\gamma\sum_{n_1=0}^p \sum_{l_1=0}^{2^{n_1}-1}c^p_{2^{n_1}+l_1} \sum_{n_2=0}^p 2^{(n_1+n_2-m)(\frac{1}{2}-H)} \sum_{l_2=0}^{2^{n_2}-1}c^{g}_{n_2,l_2}\int_{0}^{1} e_{n_1,l_1}(u)H_{n_2,l_2}(u)H_{m,k}(u)du\\
    &=b_{2^m+k}+x_1^p a_{2^m+k,0}+\sum_{n=0}^p\sum_{l=0}^{2^n-1}c^p_{2^n+l}a^p_{2^m+k,2^n+l},
\end{align*}
and
\begin{align*}
\nonumber     x^p_1&= x_0+\beta\frac{x_0+x_1^p}{2}+ \beta \sum_{n=0}^p\sum_{l=0}^{2^n-1}c^p_{2^n+l} 2^{-n(H+1)-2}-\gamma(x_1^p-x_0)\sum_{n_2=0}^p \sum_{l_2=0}^{2^{n_2}-1}2^{-{n_2}(1+H)-2}c^g_{n_2,l_2}\\
&-\gamma\sum_{n_1=0}^p\sum_{n_2=0}^p \sum_{l_1=0}^{2^{n_1}-1} \sum_{l_2=0}^{2^{n_2}-1}2^{n_1(1-H)+n_2(1-H)}c^p_{n_1,l_1} c^{g}_{n_2,l_2} \Delta^2_{2^{n_1}+l_1,2^{n_2}+{l_2}}(1)\\
\label{LE:eq4}    &:=b_0+x_1a_{0,0}+\sum_{n_1=0}^p\sum_{l_1=0}^{2^{n_1}-1}c^p_{2^{n_1}+l_1}a^p_{0,2^{n_1}+l_1}.
\end{align*}
\end{proof}

\begin{lemma}
Let $X_p$ be the solution of equation \eqref{LE:eq2}. Then $X_p$ tends to the solution of \eqref{lin:eq:def} in the norm $\|\cdot\|_{\alpha,\infty}$.
\end{lemma}
\begin{proof}
Let $X$ be the solution of \eqref{lin:eq:def}. 
Recall the operator $U(x)=\beta I_{0+}^1 x+\gamma \int_0^{\cdot}xdg, x\in \mathbf{H}^H[0,1]$ and consider the norm $\|\cdot\|_{\alpha,\lambda}$ with $\lambda>0.$ Then
\begin{align*}
\nonumber&\|X_p-X\|_{\alpha,\lambda}=\left\|\beta I^{1}_{0+}S_p X_p - \beta I^{1}_{0+}X+\gamma\int_0^\cdot S_p X_p d[ S_p g]- \gamma\int_0^\cdot X  d g\right\|_{\alpha,\lambda}\\
&\leq \left\| U (X_p) -U(X)\right\|_{\alpha,\lambda} +\left\|\beta I^{1}_{0+}[S_p X_p -X_p] + \gamma\int_0^\cdot S_p X_p d[ S_p g]- \gamma\int_0^\cdot X_p  d g\right\|_{\alpha,\lambda}\\
&\leq \left\| U (X_p) -U(X)\right\|_{\alpha,\lambda} + \left\| U (S_pX_p) -U(X_p)\right\|_{\alpha,\lambda}+\left\|\gamma\int_0^\cdot S_p X_p d[ S_p g]- \gamma\int_0^\cdot S_p X_p  d g\right\|_{\alpha,\lambda}.
\end{align*}
Since $\|U(x)-U(y)\|_{\alpha,\lambda}\leq \frac{1}{2}\|x-y\|_{\alpha,\lambda},$ then
\begin{equation}
 \label{lemma21}   \frac{\|X_p-X\|_{\alpha,\lambda}}{2}\leq \frac{\|S_p X_p-X_p\|_{\alpha,\lambda}}{2}  + |\gamma|\left\|\int_0^\cdot S_p X_p d[ S_p g -g]\right\|_{\alpha,\lambda}.
\end{equation}
By \cite[Propositions 4.2 and 4.4]{nualart}, there  exist  constants $d_1$ and $d_2$ such that the second norm in the RHS of \eqref{lemma21} is bounded above by
\begin{align*}
&\frac{d_2}{\lambda^{1-2\alpha}} \frac{1+\|S_p X_p\|_{\alpha,\lambda}}{\Gamma(1-\alpha)}\sup_{0<s<t<1}\left| D^{1-\alpha}_{t-}[S_p g- g- S_p(t)+g(t)](s)\right|\\
&\leq  \frac{d_2}{\lambda^{1-2\alpha}} \frac{1+\|S_p X_p\|_{\alpha,\lambda}}{\Gamma(1-\alpha)} \sum_{n=p}^\infty 2^{m(1-\alpha-H)}L c_1(1-\alpha)\to 0, p\to \infty,
\end{align*}
where the last inequality follows from   \eqref{ineqDTemk1}.

Similarly to \eqref{estSP1}, consider the norm
\begin{align}
\nonumber    &\|S_p X_p-X_p\|_{\alpha,\lambda}= \sup_{0<t<1}e^{-\lambda t} \left(|S_p X_p-X_p|(t)+\int_0^t\frac{|[S_p X_p-X_p](t)-[S_p X_p-X_p](s)|}{(t-s)^{\alpha+1}}ds\right)\\
\nonumber    &\leq \sup_{0<t<1}e^{-\lambda t} \left(\frac{Lt^{\alpha}}{\Gamma(1+\alpha)}\|x^H-S_p x^H\|_{\infty}+ \sum_{m=p}^\infty 2^{m(\frac{1}{2}-H)}L\int_0^t\frac{\left|\sum_{k=0}^{2^m-1}(e_{m,k}(t)-e_{m,k}(s))\right|}{(t-s)^{\alpha+1}}ds\right).
\end{align}
Consider the last integral in more detail. At first, note that \begin{align*}
    \sum_{k=0}^{2^m-1}e_{m,k}(t)&=2^{-\frac{m}{2}}\left((2^m t-k)\wedge(1+k-2^m t)\right)_+\\
    &=2^{-\frac{m}{2}}\left(2^m t-\lfloor2^m t\rfloor)\wedge(1+\lfloor2^m t\rfloor-2^m t)\right)=2^{-\frac{m}{2}}e_{0,0}(\{2^m t\}), t\in [0,1],
\end{align*}
 $|e_{0,0}(\{x\})-e_{0,0}(\{y\})|\leq 1,$ $|e_{0,0}(\{x\})-e_{0,0}(\{y\})|\leq |\{x\}-\{y\}|,x,y\geq 0.$
Therefore,
\begin{align}
\nonumber &\int_0^t\frac{\left|\sum_{k=0}^{2^m-1}(e_{m,k}(t)-e_{m,k}(s))\right|}{(t-s)^{\alpha+1}}ds=2^{-\frac{m}{2}}\int_0^t\frac{\left|e_{0,0}(\{2^m t\})-e_{0,0}(\{2^m s\})\right|}{(t-s)^{\alpha+1}}ds\\
\label{intesumkm}  &=  2^{m\alpha-\frac{m}{2}}\int_0^{2^m t}\frac{\left|e_{0,0}(\{2^m t\})-e_{0,0}(\{z\})\right|}{(2^m t-z)^{\alpha+1}}dz=  \frac{\left(\int_{0}^{\lfloor2^m t\rfloor-0.5}+\int_{\lfloor2^m t\rfloor-0.5}^{\lfloor2^m t\rfloor}+\int_{\lfloor2^m t\rfloor}^{2^m t}\right)(\ldots)dz}{2^{\frac{m}{2}-m\alpha}}.
\end{align}
The first integral in the RHS of \eqref{intesumkm} is bounded by $\int_{0}^{\lfloor2^m t\rfloor-0.5}\frac{dz}{(2^m t-z)^{\alpha+1}}\leq \frac{(2^mt-\lfloor2^m t\rfloor+0.5)^{-\alpha}}{\alpha}\leq 2^{\alpha}/\alpha.$ If $z\in (\lfloor2^m t\rfloor-0.5,\lfloor2^m t\rfloor),$ then $z=\lfloor2^m t\rfloor-1+\{z\},$ $2^m t-z=\{2^mt \}+1-\{z\}$ and the second integral in the RHS of \eqref{intesumkm} is less or equal than
\begin{align}
\nonumber&\int_{\lfloor2^m t\rfloor-0.5}^{\lfloor2^m t\rfloor}\frac{\left|\{2^m t\}\wedge(1-\{2^m t\})-(1-\{z\})\right|}{(\{2^m t\}+1-\{z\})^{\alpha+1}}dz\leq    \int_{\lfloor2^m t\rfloor-0.5}^{\lfloor2^m t\rfloor}\frac{dz}{(\{2^m t\}+1-\{z\})^{\alpha}}\\
\label{intsumeq2}&\leq \frac{1}{1-\alpha}\left((0.5+\{2^m t\})^{1-\alpha}-(\{2^m t\})^{1-\alpha}\right)\leq \frac{2^{\alpha-1}}{1-\alpha}.
\end{align}
If $z\in (\lfloor2^m t\rfloor,2^m t),$ then  $2^m t-z=\{2^m t\}-\{z\}$ and $\left|e_{0,0}(\{2^m t\})-e_{0,0}(\{z\})\right|\leq \{2^m t\}-\{z\}=2^m t-z.$ Thus, the third integral in the RHS of \eqref{intesumkm}  equals
\begin{equation}
\label{intsumeq3}
\int_{\lfloor2^m t\rfloor}^{2^m t}\frac{\left|e_{0,0}(\{2^m t\})-e_{0,0}(\{z\})\right|}{(2^m t-z)^{\alpha+1}}dz\leq \int_{\lfloor2^m t\rfloor}^{2^m t}\frac{dz}{(2^m t-z)^{\alpha}}=\frac{(\{2^m t\})^{1-\alpha}}{1-\alpha}\leq \frac{1}{1-\alpha}.
\end{equation}
Hence, we get from \eqref{intsumeq2} and \eqref{intsumeq3} that the upper bound for the right hand side  of \eqref{intesumkm}  is
$2^{m\alpha-\frac{m}{2}}\left(\frac{2^{\alpha}}{\alpha}+\frac{2^{\alpha-1}+1}{1-\alpha}\right).$ Finally, the norm $\|S_p X_p-X_p\|_{\alpha,\lambda}$ is bounded by
\begin{align}
\label{lemma:p15}  &\sup_{0<t<1}e^{-\lambda t} \left(\frac{Lt^{\alpha}}{\Gamma(1+\alpha)}\|x^H-S_p x^H\|_{\infty}+ L\left(\frac{2^{\alpha}}{\alpha}+\frac{2^{\alpha-1}+1}{1-\alpha}\right) \sum_{m=p}^\infty 2^{m(\alpha-H)}\right).
\end{align}
Note that $\alpha<H$ in equation \eqref{lin:eq:def}.
Therefore, the right hand side of \eqref{lemma:p15} tends to 0 as $p\to \infty.$
It was shown in the proof of Lemma \ref{lemma:cont1} that the first term in  \eqref{lemma21} tends to 0 as $p\to \infty.$
Thus, $\|X_p-X\|_{\alpha,\infty}\to 0, p\to \infty.$
\end{proof}

\subsection{Numerical experiments: the Volterra integral  equation }
In this section we illustrate our method of solution of  \eqref{Lang:eq:def} by numerical examples.

Let $0<H<\alpha\in (0,1)$ and put $g(t)=t^H(1-t^\alpha),\;t\in [0,1].$ Then the solution of equation $X(t)=\frac{\Gamma(\alpha+H+1)}{\Gamma(H+1)}I_{0+}^\alpha X(t) +t^H(1-t^\alpha),t\in [0,1]$ obviously equals $\{X(t)=t^H,t\in[0,1]\}$.

We solve truncated equation \eqref{Lang:eq2} by Lemma \ref{Lang:sol} for several combinations of $\alpha$ and $ H$. For each case we compute the norm of the error $\|X-X_p\|_{\infty},$ where $X_p$ is the solution of truncated equation, and present them on Table \ref{tb1}.  

\begin{table}[]
\footnotesize
    \centering
\begin{tabular}{lllllllll}
& $H$=0.01& $H$=0.2 & $H$=0.2 & $H$=0.2 & $H$=0.5& $H$=0.5 & $H$=0.8 & $H$=0.8 \\
$p$  & $\alpha$=0.05 &  $\alpha$=0.3 & $\alpha$=0.5 &  $\alpha$=0.8 & $\alpha$=0.51 &$\alpha$=0.8 & $\alpha$=0.81 &  $\alpha$=0.9 \\
\hline
3  & 2.33e-01       & 6.76e-02     & 5.83e-02     & 5.04e-02     & 2.38e-02      & 2.04e-02     & 5.60e-03      & 5.39e-03     \\
4  & 1.92e-01       & 4.32e-02     & 2.66e-02     & 2.25e-02     & 9.02e-03      & 7.56e-03     & 1.78e-03      & 1.71e-03     \\
5  & 1.62e-01       & 2.83e-02     & 1.53e-02     & 9.96e-03     & 3.34e-03      & 2.76e-03     & 5.50e-04      & 5.28e-04     \\
6  & 1.39e-01       & 1.89e-02     & 9.16e-03     & 4.37e-03     & 1.23e-03      & 9.97e-04     & 1.68e-04      & 1.61e-04     \\
7  & 1.21e-01       & 1.28e-02     & 5.53e-03     & 1.91e-03     & 5.97e-04      & 3.58e-04     & 5.06e-05      & 4.85e-05     \\
8  & 1.07e-01       & 8.75e-03     & 3.35e-03     & 8.35e-04     & 2.92e-04      & 1.28e-04     & 1.51e-05      & 1.45e-05     \\
9  & 9.48e-02       & 6.02e-03     & 2.04e-03     & 3.64e-04     & 1.43e-04      & 4.54e-05     & 4.50e-06      & 4.31e-06     \\
10 & 8.50e-02       & 4.17e-03     & 1.25e-03     & 1.71e-04     & 7.07e-05      & 1.61e-05     & 1.33e-06      & 1.27e-06
\end{tabular}
    \caption{Volterra integral equation: norms of the error $\|X-X_p\|_{\infty}$}
    \label{tb1}
\end{table}

\subsection{Numerical experiments: linear integral  equation }
In this section we consider the numerical solution of \eqref{lin:eq:def}.

First, we put  $g(t)=0.5^H - |t-0.5|^H,\,t\in [0,1]$ for $H\in (0.5,1),$ and $\beta=-2,$ $\gamma=3, x_0=1$ in \eqref{lin:eq:def}. We take $p=6,$ $H\in\{0.51,0.6,0.7,0.8,0.9\},$ solve truncated equation \eqref{LE:eq2} by Lemma \ref{LE:sol} and get the Takagi-Landsberg representation of the truncated solution $X_p$ with coefficients $x_0,x_1^p,c_p$. We present the values of the error's norm $\|X-X_p\|_{\infty}$ in Table \ref{tb2}, where $X(t)=x_0\exp(\beta t+\gamma g(t)),t\in[0,1]$ is the exact solution. Moreover, we compute the difference between the exact coefficients $x_1,c^x$ in the representation of $X$ and $x^p_1,c^p.$ The values of $\max_{1\leq n\leq2^{p+1}}|c^x_n-c^p_n|$ are given in Table \ref{tb2}.


\begin{table}[]
\footnotesize
    \centering
\begin{tabular}{lllllll}
& $H$=0.51& $H$=0.6 & $H$=0.7 & $H$=0.8 & $H$=0.9& $H$=0.99 \\
\hline
$\|X-X_p\|_{\infty}$ & 0.18934 & 0.08398 & 0.03218 & 0.01142 & 0.00325 & 0.00047   \\
$\max_{1\leq n\leq2^{p+1}}|c^x_n-c^p_n|$  &  0.03701 & 0.01305 & 0.00409& 0.00124 & 0.00043 & 0.00028
\end{tabular}
    \caption{Linear equation: Description of the error $X-X_p$}
    \label{tb2}
\end{table}

Second, we illustrate our method with the function $g(t)=\sum_{m=0}^7\sum_{k=0}^{2^m-1}c^g_{m,k} e_{m,k}(t),t\in[0,1],$ where $c^g$ are  some bounded coefficients (we simulate them randomly). The example of function $g,$  the corresponding exact $X$ and truncated $X_p$ $(p=6)$  solutions of \eqref{lin:eq:def} with $H=0.51$ are presented on Figure \ref{Fig1}. One can observe that the small difference between the exact and truncated solution.
Moreover, if we increase the value of $p=7,$ then the graphs of $X$ and $X_p$ for $H=0.501$ become visually indistinguishable and the computed norm of the error $\|X-X_p\|_{\infty}$ is 0.01888 for this example.

From the other hand, the wrong value of $H,$ which is greater than the H\"{older} exponent of $g,$ affects on solution $X_p$ and the error between $X_p$ and $X$ increases. We illustrate such mis-specification of $H$ on Figure \ref{Fig2}, where one clearly see the difference between the exact solution $X$ and numerical solution $X_p$ when $H$ is significantly larger than true value 0.5.
\begin{figure}[h]
\begin{minipage}{0.48\linewidth}
\centering
\includegraphics[width=\linewidth]{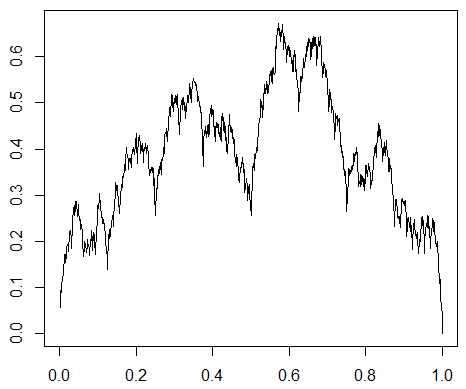}
\end{minipage}
\begin{minipage}{0.48\linewidth}
\centering
\includegraphics[width=\linewidth]{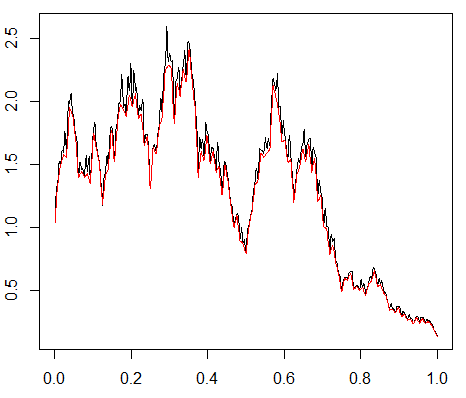}
\end{minipage}
\label{Fig1}
\caption{Left: function $g.$ Right: solutions $X$ (black) and $X_p$ (red) for $H=0.51$.}
\end{figure}
\begin{figure}
\begin{minipage}{0.48\linewidth}
\centering
\includegraphics[width=\linewidth]{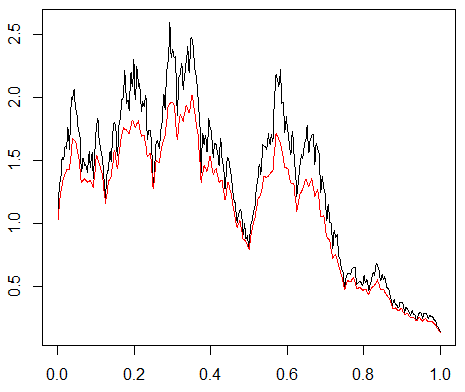}
\end{minipage}
\begin{minipage}{0.48\linewidth}
\centering
\includegraphics[width=\linewidth]{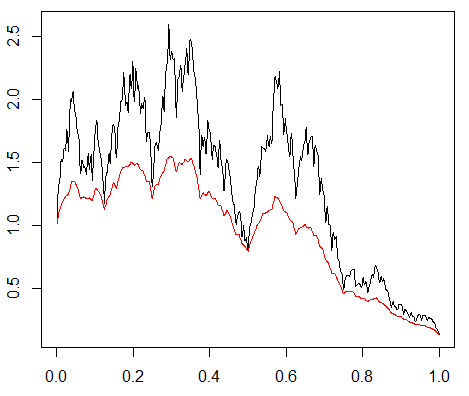}\\
\end{minipage}
\caption{The mis-specification of $H=0.6$ (left) and $H=0.8$ (right): graphs of $X$ (black) and $X_p$ (red).}
\label{Fig2}
\end{figure}
\bibliographystyle{abbrv}
\bibliography{Lib}
\end{document}